\documentclass[a4paper, 11pt]{article}
\usepackage{amsmath,amsfonts,amsthm,amssymb,amscd,graphicx,array}
\usepackage{hyperref} 
\usepackage{color} 

\newtheorem{thm}{Theorem}[section]
\newtheorem*{thm*}{Theorem}
\newtheorem{lemma}[thm]{Lemma}
\newtheorem{prop}[thm]{Proposition}

\newtheorem{corr}[thm]{Corollary}
\theoremstyle{definition}

\newtheorem{exmple}[thm]{Example}

\theoremstyle{remark}
 
\newtheorem*{rmq}{\textit{Remark}}
\newtheorem{rmk}[thm]{\textit{Remark}}
\newtheorem{rmks}[thm]{\textit{Remarks}}
\renewcommand{\proof}{\noindent\textit{Proof}\/: \,\,}
\newenvironment{diagram}{
\begin{matrix}}{\end{matrix}}
\newcommand{\C}{{\mathbb{C}}}
\newcommand{\Q}{{\mathbb{Q}}}
\newcommand{\R}{{\mathbb{R}}}

\newcommand{\bF}{{\mathbb{F}}}
\newcommand{\Z}{{\mathbb{Z}}}

\newcommand{\bP}{{\mathbb{P}}}
\newcommand{\PP}{{\mathbb{P}}}

\newcommand\germ[1]{{\mathfrak{#1}}}
\newcommand{\eh}{\germ h}%
\newcommand{\eo}{\mathfrak{o}}%
\newcommand{\es}{\mathfrak{s}}%
\newcommand{\eu}{\mathfrak{u}}%
\newcommand{\cA}{\mathcal{A}}

\def\cl#1#2{\text{\rm Cl}^ {#1}({#2})}
\newcommand{\comp}{\raise1pt\hbox{{$\scriptscriptstyle\circ$}}}

\def\half{\frac 12}
\def\id{\mathop{\rm id}\nolimits}
\def\ii{{\rm i}}
\def\lwedge{\mathsf{\Lambda}}
\def\lset{\{}  
\def\rset{\}}  
\def\set#1{\lset#1\rset} 
\def\st{\mid}   
\def\sett#1#2{\lset #1 \st #2 \rset}  
\renewcommand\setminus{-}
\def\qf#1{\langle #1 \rangle}
\def\tr{\mathop{\rm Tr}\nolimits}
\newcommand\Tr{{}^{\mathsf{T}}\kern-0.9pt} 
\def\vrulecup{\kern1.30pt\cup\kern-4.15pt\vrule height 6.0pt depth 0pt}
 \def\mapright#1{\mathop{\vbox{\ialign{
                                ##\crcr
    ${\scriptstyle\hfil\;\;#1\;\;\hfil}$\crcr
 \noalign{\kern2pt\nointerlineskip}
    \rightarrowfill\crcr}}\;}}

\def\mapleft#1{\mathop{\vbox{\ialign{
                                ##\crcr
    ${\scriptstyle\hfil\;\;#1\;\;\hfil}$\crcr
 \noalign{\kern2pt\nointerlineskip}
    \leftarrowfill\crcr}}\;}}

\newcommand\rarrow[3]{\smash{\mathop{\hbox to#3{\rightarrowfill}}\limits
^{\scriptstyle#1}_{\scriptstyle#2}}}
 
\newcommand\larrow[3]{\smash{\mathop{\hbox to#3{\leftarrowfill}}\limits
^{\scriptstyle#1}_{\scriptstyle#2}}}
 
\newcommand\darrow[3]{\llap{$\scriptstyle #1$}
\left\downarrow\vbox to#3{}\right.\rlap{$\scriptstyle #2$}}

\def\into{\hookrightarrow}

\newcommand\aut{\operatorname{Aut}}

\newcommand\End{\mathop{\rm End}\nolimits}

\newcommand\gl[1]{\operatorname{GL}({#1})}

\newcommand\im{\operatorname{Im}}

\newcommand{\NS}{\mathop{\rm NS}}
\newcommand{\MWL}{\mathop{\rm MWL}}
\newcommand\ogr[2]{\operatorname{O}^{#2}({#1})}

\def\rank{\operatorname{rank}}
\newcommand\slgr[1]{\operatorname{SL}({#1})}
\newcommand\so[2]{\operatorname{SO}^{#2}({#1})}
\newcommand\smpl[1]{\operatorname{Sp}({#1})}
\newcommand\su[1]{\operatorname{SU}({#1})}
\newcommand\ugr[2]{\operatorname{U}^{#2}({#1})}

\title{Abelian Fourfolds of Weil type and certain K3 Double Planes}
\author{Giuseppe Lombardo, Chris Peters and Matthias Sch\"utt}

\date{April 9  2013}

\begin{document}
\maketitle

\begin{abstract}
Double planes branched in 6 lines give a famous example of K3 surfaces.
Their moduli are well understood and related to Abelian fourfolds of Weil type.
We compare these two moduli interpretations
and in particular divisors on the moduli spaces.
On the K3 side, this is achieved with the help of elliptic fibrations.
We also study the Kuga-Satake correspondence on these special divisors.
\end{abstract}
 
\section*{Introduction}

The wonderful geometry of double planes branched in 6 lines has been addressed extensively in the
articles \cite{MSY,Mats}. When the lines are in general position such double planes are K3 surfaces and their moduli are well known and are determined through the period map. 

The target space for the period map is defined as follows. Consider the lattice
\begin{equation}\label{eqn:LatticeT}
T=U^2\perp \qf{-1}\perp \qf{-1}.
\end{equation}
Here $U$ is the hyperbolic lattice with basis $\set{e,f}$ and $\qf{e,e}=\qf{f,f}=0$, $\qf{e,f}=1$ and for $n\in\Z$ the notation $\qf{n}$ means that we have  a one dimensional lattice with basis $\set{g}$ such that $\qf{g,g}=n$. We write $z=(z_1,z_2,z_3,z_4)\in T\otimes\C$ and form the associated domain 
\[
D(T):= \sett{z\in \bP(T\otimes\C)}{\qf{z,z}=0,\, \qf{z,\bar z}>0}
\]
The orthogonal group $\ogr{T}{}$ acts on $D(T)$, a space with two connected components; the  commutator subgroup $\so{T}{*}$  of $\ogr{T}{}$ preserves the connected components, distinguished by the sign of $\im(z_3/z_1)$. Let  $ \mathbf{D}_4$ be the one with the positive sign. The target of the period map is
\begin{equation}\label{eqn:ModSpace}
\mathbf{M}=\mathbf{D}_4/\so{T}{*}.
\end{equation}
 This is the principal moduli space in this article.

The source of the period map is the configuration space $U_6$ of 6 ordered  lines in good  position.\footnote{This means that the resulting sextic has only ordinary nodes.}
The  linear group $\gl {3;\C}$ acts on this space and we form the quotient
\[
\begin{array}{lcl}
\mathbf{X} :&=& U_6 /\gl{3;\C}: \text{\bf configuration space of 6 ordered lines}\\
&&\hspace{6.5em} \text{\bf   in good  position in $\bP^2$.}
\end{array}
\]
The symmetric group $\germ S_6$ acts on this space  and  the quotient $\mathbf{X} $ is the configuration space of 6 non-ordered lines in good position. One of the peculiarities of the number 6 is that there is a  further  
(holomorphic) involution  $*$ on 6 lines which comes from the  
correlation map  with respect to a conic in the plane (see \S~\ref{subsec:6Lines}) and which we call \emph{the correlation involution}. 
Six-tuples of lines in good position related by this involution correspond to isomorphic double covering K3-surfaces.
The explanation of why this is true\footnote{The referee urged us to find an explanation of this since  it could not be found   in the literature.}  uses beautiful classical geometry related to Cremona transformations for which we gratefully acknowledge the source \cite{Do-Ort}.
See Prop.~\ref{Cremona} and Remark~\ref{LinesAndCremona}.

Furthermore, the involution commutes with the action of $\germ S_6$.     
So the period map descends to this quotient and  by \cite{MSY}  this yields   a biholomorphic map
\[
\mathbf{X} /H\mapright{\simeq} \mathbf{M},\quad  H:= \germ S_6  \times \set{*}\ 
\]
In other words we may identify our moduli space with the quotient of the configuration space of 6 unordered lines in good
 position by the correlation involution. 
 
 In loc.~cit.\    some special divisors on $\mathbf{X} $ (and hence on $\mathbf M$) are studied,
in particular  the divisor called $XQ$ which corresponds to 6 lines all tangent to a common conic $C$. The associated double covers are  K3 surfaces with rich geometry. Indeed this 3-dimensional variety is the moduli space $\mathbf M_{2,2}$ of  curves of genus $2$ (with a level 2 structure). To see this note that the conic $C$ meets the ramification divisor $R$ (the 6 lines) in 6 points and on the K3 surface  this gives a genus 2 curve $D$ lying over it. In \cite[\S 0.19]{MSY} it is explained that  the Kummer surface $J(D)/\set{\pm 1}$ associated to the jacobian $J(D)$ of  $D$  is isomorphic to the double cover branched in $R$ and that there is a natural level 2 structure on $J(D)$. Furthermore,   the period map induces a biholomorphism
\begin{equation}\label{eqn:XQiso}
\mathbf M_{2,2} = XQ \mapright{\simeq} (\eh_2 / \Gamma(2))^0 \into \mathbf{X} ,\quad \Gamma= \smpl{2,\Z}/\pm 1,
\end{equation} 
where the image is Zariski-open in $\eh_2 / \Gamma(2)$ (in fact, in loc. cit. an explicit partial compactification $\bar {\mathbf{X} }$ of $X$ is described such that the closure of $\mathbf M_{2,2}$ in $\bar {\mathbf{X} }$ is exactly the space $\eh_2 / \Gamma(2)$).

In \cite{MSY, Her} one also finds a relation with a certain moduli space of Abelian 4--folds. We recall that principally polarized Abelian varieties of dimension $g$ (without any further restriction) are parametrized by the Siegel upper half space $\eh_g$.  For a  generic such Abelian variety $A$ its rational endomorphism ring   $\End_\Q A$ is isomorphic to $\Q$, but special Abelian varieties have larger endomorphism rings and are parametrised by subvarieties of $\eh_g$. Here we consider $4$--dimensional $A$  with $\Z(\ii)\subset \End(A)$ where the  endomorphisms are assumed to preserve the principal polarization. 
These are   parametrized by the $4$-dimensional domain \footnote{As usual, for a matrix $W$, we abbreviate $W^*=\Tr \bar W$.}
 \[
\mathbf{H}_2:= \left\{ W\in M_{2\times 2}(\C) \mid  \frac{1}{2\ii}(W-W^*)>0\right\} \simeq \ugr{(2,2)}{}/\ugr{2}{}\times\ugr{2}{} 
\]
which is indeed (see e.g. \cite[\S1.2]{Her})  a subdomain of the Siegel upper half space $\eh_4$ by means of the embedding 
\begin{equation}
\label{eqn:H2andSiegel}
\left.\begin{array}{lcl}
\iota: \mathbf{H}_2 & \into & \eh_4\\
W& \mapsto& U\begin{pmatrix} W&0 \\0 & \Tr W \end{pmatrix} U^*,\\
 &&U= \frac {1 }{\sqrt{2}} \begin{pmatrix} \ii\mathbf{1}_2& -\ii \mathbf{1}_2\\
 \mathbf{1}_2 & \mathbf{1}_2\end{pmatrix}.
\end{array}\right\}
\end{equation}
As for all polarized Abelian varieties,  the  
complex structure of such Abelian fourfolds $A$ is faithfully reflected in the polarized weight one Hodge structure on  $H^1(A)$. The second cohomology carries an induced  weight $2$ Hodge structure. The extra complex structure on $A$  (induced by multiplication with $\ii$)   
makes it possible to single out a rational sub Hodge structure $T(A)\subset H^2(A)$ of rank $6$ with Hodge numbers $h^{2,0}=1$, $h^{1,1}=4$. Such Hodge structures are classified by points in the domain $\mathbf{D}_4$ we encountered before.  Indeed, there is an 
isomorphism
\[
\mathbf{H}_2 \mapright{\sim} \mathbf{D}_4   \simeq \ogr{2,4}{} /\ogr{2}{}\times \ogr{4}{}
\]
which is induced by an  isomorphism between classical real Lie groups which on the level of Lie algebras gives the well known isomorphism  $\es\eu(2,2)\simeq \es\eo(6)$. 
This isomorphism is implicit in \cite{MSY}. One of the aims of this article is to review this using the explicit and classically known isomorphism between corresponding Lie groups. See  \S~\ref{sec:IsoCl}; the corresponding Hodge theoretic discussion is to be found  in \S~\ref{sec:ModAV}. This gives an independent and coordinate free presentation of the corresponding results in \cite{MSY,Mats}.  In particular, on the level of moduli spaces we find (see \S~\ref{ssec:isoms}):
\[
 \mathbf{M}:=  \mathbf{H}_2/\ugr{ (2,2)  ;\Z[\ii]}{*}   \mapright{\simeq}   \mathbf{D}_4/\so{T}{*}.
\]
The group $\ugr{ (2,2)  ;\Z[\ii]}{*}  $ is an  extension of the unitary group $\ugr{(2,2);\Z[\ii]}{}  $  (with coefficients in the Gaussian integers)  by an involution as explained  below (see eqn.~\eqref{eqn:ExtU22}).

One of the new results in our paper is the study of the N\'eron-Severi lattice using the geometry of elliptic pencils. It allows us to determine the generic N\'eron-Severi lattice (in \S~\ref{ss:intro}) as well as the generic transcendental lattice. And indeed, we find $T(2)$ for the latter (here the  brackets mean that we multiply the form by $2$; this gives an even form as it should):
\begin{thm*}[=Theorem~\ref{MainTHGenericCase}] For generic $X$ as above we have for the  N\'eron-Severi lattice $\NS(X) = U\perp D_6^2\perp \qf{-2}^2$ and for the transcendental lattice $T(X) = T(2)$.\footnote{See also \cite[Prop. 2.3.1]{MSY}} 
\end{thm*}
Here $D_k$ is the lattice for the corresponding Dynkin diagram. See the notation  just after the introduction.

Of course, the divisors in $X$ parametrizing special line configuration also have a moduli interpretation on the Abelian 4-fold side. This has  been studied by Hermann in \cite{Her}. To compare the results from  \cite{MSY} and \cite{Her} turned out to be a non trivial exercise (at least for us). This explains why we needed several details from  both  papers. We collected them in \S~\ref{sec:IsoCl} and \S~\ref{ssec:SAV}. We use this comparison in particular to relate (in  \S~\ref{subsec:6Lines}) Hermann's divisors $D_\Delta$ for small $\Delta$ to some of the divisors described in \cite{MSY}.  We need this in order to describe (in \S~\ref{ssec:SpecLines}) the geometry of the corresponding K3 surfaces in more detail. The technique here is the study of the degeneration of a carefully chosen elliptic pencil (when the surface moves to the special divisor) which reflects the corresponding lattice enhancements explained in \S~\ref{ssec:LattEnh}.  This technique enables us to calculate the N\'eron-Severi and transcendental lattice of the generic K3 on the divisors  $D_\Delta$ for $\Delta=1,2,4,6$ respectively (we use Hermann's notation). 

For the full statement we refer to Theorem~\ref{MainTHDegenerations}; here  we want to single out the result for $\Delta=1$:

\begin{thm*}  The divisor $D_1\subset \mathbf{M}$ corresponds to the divisor $XQ\subset \mathbf{X} $ and so (see \eqref{eqn:XQiso}) is isomorphic to a Zariski-open subset in the  moduli space of genus two curves. The generic point on $D_1$ corresponds to a K3 surface  which is a double cover of the plane branched in 6 lines tangent to a common conic. Its N\'eron-Severi lattice is $\NS(X_1)=U\perp D_4\perp D_8 \perp A_3$ and its transcendental lattice is $U(2)^2\perp\qf{-4}$. 
\end{thm*}

This requires a little further explanation beyond what is stated in Theorem~\ref{MainTHDegenerations}:  we have seen in \eqref{eqn:XQiso} that the image of $XQ$ in
$\mathbf{M}=\mathbf{X} /H$  we get a variety isomorphic to  a Zariski-open subset of $\eh_2/ \Gamma$, the moduli space of principally polarized Abelian varieties of dimension $2$ (=the moduli space of genus 2 curves).

A further novelty of this paper is the role of the Kuga-Satake correspondence.
For the general configuration of lines this has been done by one of us in \cite{Lom},
building on work of Paranjape \cite{Par}.  
One of these   earlier results reviewed  in \S~\ref{sec:KS} (Theorem~\ref{Lombardo})  states  that  the Kuga-Satake construction    gives back the original Abelian $4$--fold up to isogeny.  In the present  paper we  explain   what this construction specializes to for the K3 surface  on the generic $D_\Delta$, this time for  \emph{all} $\Delta$. See Theorem~\ref{KSSpecialCase}.
 
\section*{Notation}
The bilinear form on a lattice is usually denoted by $\qf{-,-}$. Several standard lattices as well as standard conventions are used:
\begin{itemize}
\item Orthogonal direct sums of lattices is denotes by $\perp$;
\item For a lattice $T$ the orthogonal group is denoted $\ogr{T}{}$ its subgroup of commutators  is a subgroup of the special orthogonal group   $\so{T}{}$ and is denoted   $\so{T}{+}$;
\item   Let $(T,\qf{-,-})$ be a lattice. The  dual of $T$ is defined by 
\[
T^*:= \sett{x\in T\otimes\Q}{ \qf{x,y}\in\Z,\quad \text{for all } y\in T.}
\]
Note that $T\subset T^*$. The  discriminant group  $\delta(T)$ is the finite Abelian group $T^*/T$. If $T$ is even, i.e. $\qf{x,x}\in 2\Z$, the form $\qf{-,-}$ induces a $\Q/\Z$-valued bilinear form $b_T$ on $\delta(T)$ with associated  $\Q/2\Z$-valued quadratic  form $q_T$.
\item The group  of matrices with values in a subring  $R\subset \C$ preserving the hermitian form with Gram matrix
\begin{equation}
\label{eqn:deff}
\mathbf{J}:=   \begin{pmatrix} \mathbf{0}_2 & \mathbf{1}_2 \\
-\mathbf{1}_2 & \mathbf{0}_2 \end{pmatrix}.
\end{equation}
will be denoted $\ugr{(2,2); R}{}$ because it has signature $(2,2)$;
\item For $a\in \Z$ the lattice $\qf{a}$ is the 1-dimensional lattice with basis $e$ and $\qf{e,e}=a$;
\item If $L$ is a lattice, the lattice $L(a)$ is the same $\Z$--module, but the form gets multiplied by $a$; 
\item $U$ is the standard hyperbolic lattice with basis $\set{e,f}$ and $\qf{e,e}=\qf{f,f}=0$, $\qf{e,f}=1$; 
\item  $A_k$, $D_k$, $E_k$: the standard \emph{negative definite} lattices associated to the Dynkin diagrams: if the diagrams have vertices $v_1,\dots,v_k$, we put $-2$  on the diagonal, and in the entries $ij$ and $ji$ we put $1$
 or $0$ if $v_i$ and $v_j$ are connected or not connected respectively:
 \end{itemize}
  \setlength{\unitlength}{.5mm}
\begin{picture}(90,70)(0,-20)
\multiput(73,38)(20,0){6}{\circle*{3}}
\put(73,38){\line(1,0){80}}
\multiput(153,38)(4,0){5}{\line(1,0){2}}
\put(178,38){$A_k$}

\multiput(23,8)(20,0){4}{\circle*{3}}
\multiput(23,8)(20,0){4}{\circle*{3}}
\put(23,8){\line(1,0){20}}
\put(63,8){\line(1,0){20}}
\put(83,8){\line(2,1){17}}
\put(83,8){\line(2,-1){17}}
\put(100,16){\circle*{3}}
\put(100, 0){\circle*{3}}
\multiput(43,8)(4,0){5}{\line(1,0){2}}
\put(94,8){$D_k$}
 \multiput(123,8)(20,0){6}{\circle*{3}}
\put(123,8){\line(1,0){80}}
\put(143,8){\line(0,-1){20}}
\put(143,-12){\circle*{3}}
\multiput(203,8)(4,0){5}{\line(1,0){2}}
\put(226,8){$E_k$}
 \end{picture}

Furthermore, for a projective surface $X$ we let $\NS(X)\subset H^2(X;\Z)$ and $T(X)= \NS(X)^\perp$ be the N\'eron-Severi lattice, respectively the  transcendental lattice equipped with the lattice structure from $H^2(X;\Z)$, i.e. the \ intersection product.

Finally we recall the convention to denote congruence subgroups. Suppose $R$ is a subring of $\C$, and $V_R$ a free $R$-module of finite rank and $G$ a subgroup of the group $\aut V_R$. For any principal ideal $(\omega)\subset R$ we set 
\[
G (\omega):= \sett{g\in G}{g\equiv \id \mod (\omega)}.
\] 
In what follows we restrict ourselves to
\[
R=\Z[\ii]\subset \C,\quad \omega=1+\ii,\quad G= \su{ (2,2 );\Z[\ii]}.
\]

\section{Two Classical Groups, Their Associated Domains and Lattices} \label{sec:IsoCl}
\subsection{The Groups}\label{ssec:isoms}
 We summarize some classical results from \cite[IV. \S~8]{Die}.

Let $K$ be any field and $V$ a $4$-dimensional $K$--vector space. The decomposable elements in $W:=\lwedge^2V$ correspond to the $2$-planes in $V$;
the corresponding points in  $\bP(W)$ form  the quadric $G$ which is the image of the Grassmann variety 
  of $2$-planes in $V$  under the Pl\"ucker-embedding. Concretely, choosing a basis $\set{e_1,e_2,e_3,e_4}$ and setting 
\begin{equation}
\label{eqn:qdefn}
x\wedge y = q(x,y) \,e_1\wedge e_2\wedge e_3\wedge e_4,\quad x,y\in W,
\end{equation}
the quadric $G$ has equation $q(x,x)=0$. 

On $G$ we have two types  of planes: the first  corresponds to planes contained in a fixed hyperplane of $V$, the second type  corresponds to planes passing through a fixed line. Both types of planes therefore correspond to $3$-dimensional subspaces of $W$ on which the bilinear form $q$ is isotropic. Its  index, being the maximal  dimension of $q$--isotropic subspace equals $3$. 

Any $A\in \gl V$ induces a linear map $B=\lwedge^2 A $ of $W$ preserving the quadric $G$. Conversely such a linear map $B$  is of this form, provided it preserves the two types of planes on the quadric $G$. This is precisely the case if $\det (B)>0$ and so one obtains the classical  isomorphism  between  \emph{simple} groups
\begin{eqnarray*}
\slgr {4,K}  /\text{center} & \mapright{\simeq}&  \so{6,q,K}{+}  /\text{center}, \\
\qquad  \quad  A &\mapsto &\lwedge ^2A 
\end{eqnarray*}
where we recall that the superscript $+$ stands for  the commutator subgroup.\footnote{The last subgroup can also be identified with the subgroup of elements whose spinor norm is $1$, but we won't use this characterization.}

From this isomorphism several others are deduced (loc. cit) through a process of field extensions. 
The idea is that if $K=k(\alpha)$, an imaginary   quadratic extension of a real field $k$, the form $q$ which over $K$ has maximal index $3$, over $k$ can be made to have  index $2$.  This is done as follows. One restricts to a subset of  $K$-linear transformations of $V$ which preserve a certain well-chosen anti-hermitian form $f$. The linear maps $\lwedge^2 A$ then  preserve the  \emph{hermitian} form $g=\lwedge^2 f$ given by
\begin{equation}\label{eqn:FormForg}
g(x\wedge y,z\wedge t)=  \det \begin{pmatrix} 
f(x,z) & f(y,z)\\
f(x,t) & f(y,t)
\end{pmatrix}
\end{equation}
Suppose that in some $K$--basis for $W$ the Gram matrices for $g$ and $q$ coincide and both have entries in $k$. Then the matrix of $B=\lwedge^2 A$ being at the same time $q$--orthogonal and $g$--hermitian must be real. So this yields  an isomorphism
\begin{equation*} 
\begin{array}{lcl}
  \su{V,f,K}/\text{center}   & \mapright{\simeq} &   \so{6,q,k} {+}/\text{center} \\
\qquad  \quad  A &\mapsto &\lwedge ^2A 
\end{array}
\end{equation*}
 In our situation $K=k(\ii)$ (with, as before, $k$ a real field). We choose our basis $\set{a_1,a_2,b_1,b_2}$, for $V$ in such a way that the anti-hermitian form $f$ has Gram matrix 
$\ii \mathbf{J}$ (see \eqref{eqn:deff}).

The  Gram matrix of  $-q$ is the (integral) Gram matrix  for  $U\perp U\perp  U$. The Gram matrix of   the hermitian form $-g$ is  found to be the  Gram matrix  for  $U\perp U\perp \qf{ -1}\perp \qf{ -1}$  of signature $(2,4)$.   In a  different $K$-basis for $W$   the Gram matrix of $q$ is found to coincide  with the Gram  matrix for $g$:
 \begin{lemma}  Let $\omega= \half (-1+\ii)$ and set 
\[
\begin{matrix}
{g}_ 1=f_1, \quad    {g}_ 2=f_2,   \quad   {g}_ 3=f_3,  \quad {g}_ 4=f_4,  &  g_ 5=\omega f_5-  \bar \omega   f_6, &   g_ 6=    ( -\bar\omega  f_5+\omega   f_6)\\
\hspace{-5em} T   : =    \Z \text{--lattice  spanned by }  &\hspace{-13em} \set{g_1,\dots,g_6}.
\end{matrix}
\]
Then the Gram matrices of $-q$ (see \eqref{eqn:qdefn}) and $-g$ (see \eqref{eqn:FormForg}) on  $T $  are both equal to $U\perp U \perp \qf{  -1} \perp  \qf{  -1} $. 
 \end{lemma}
Indeed, in this basis    we obtain the desired isomorphisms:
\begin{lemma}[\protect{\cite[\S~1.4]{Mats}}] \label{ClassicalIsom} Set  $T_\Q:=T\otimes \Q$. We have an isomorphism of $\Q$--algebraic groups (see \eqref{eqn:deff})
\begin{equation*} 
\begin{array}{lcl}
  \su{(2,2) ;\Q(\ii)}/\set{\mathbf{1},-\mathbf{1} }  & \mapright{\simeq} &   \so{T_\Q;\Q} {+}  \\
\qquad  \quad  A &\mapsto &\lwedge ^2A|T_ \Q .
\end{array}
\end{equation*}
On  the level of integral points we have
\begin{equation}\label{eqn:FundIso}
\varphi: \su{(2,2) ;\Z[\ii]}/ \set{\mathbf{1},-\mathbf{1} } \mapright{\sim}\so{T}{+}.
\end{equation}
This   isomorphism   induces    an isomorphism   of real Lie groups
\[
\su{(2,2);\C/}\set{\mathbf{1},-\mathbf{1} }  \mapright{\simeq} \so{(2,4);\R}{+}.
\]
The target  is the component of the identity of $\so{(2,4);\R}{}$ and is a simple group.
\end{lemma}
\proof
 As noted before, a matrix which is at the same time hermitian and orthogonal with respect to the same real matrix has to have real coefficients.  
 So the map $A\mapsto \lwedge^2 A$ sends $\su{(2,2) ;\C}$ injectively to a connected real subgroup of $\so{(6,q)}{}$. A dimension count shows that we get the entire connected component of the latter group which is (isomorphic to) $\so{(2,4);\R}{+}$.
 
Assume now that $A$ has coefficients in $\Q(\ii)$. It then also  follows that $\lwedge ^2A|T_ \Q$  must have rational coefficients, i.e.  we have shown the first  assertion of the lemma. The assertion about integral points follows since the change of basis matrix from the $f$-basis  to the $g$-basis is unimodular and hence if $A$ preserves a lattice, $\lwedge^2 A|T $ preserves the corresponding lattice.
 \qed\endproof
 
 \begin{rmk} \label{ExtIso} This isomorphism can be extended to $\ugr{(2,2) }{}$ modulo its center $\ugr{1}{}$ provided one  takes  the semi-direct product of  the latter group with an involution $\tau $ which acts on matrices $A\in \ugr{(2,2) }{}$ by $\tau(A)=
 \bar A$.     We set
 \begin{equation}\label{eqn:ExtU22}
 \ugr{(2,2) } {*}:= \ugr{(2,2) }{}\rtimes    \langle \tau \rangle.
\end{equation} 
Now  $\tau$ also induces complex conjugation on $\lwedge^2 V$ with respect to the real structure given by the real basis $\set{f_1,\dots, f_6}$. This involution preserves $\set{g_1,\dots,g_4}$ but interchanges $g_5$ and $g_6$. So on $T$ the involution becomes identified with the involution\begin{equation}\label{eqn:TauTilde}
\tilde\tau :T \to T,\quad \tilde\tau(g_k)= g_k, \, k=1,\dots,4, \,\,\tilde\tau(g_5)= g_6 
\end{equation}
 and  one   can extend the homomorphism $\varphi$ from  \eqref{eqn:FundIso} by sending $\tau$ to $\tilde \tau$. 

Accordingly, we 
 define a two component subgroup of $\ogr{2,4}{}$:
\[
\so{2,4}{*}= \so{2,4}{+} \rtimes   \qf{  \tilde \tau }.
\]
Note that $-\mathbf{1} \in \so{2,4}{*}$ so that 
\[
\begin{array}{ll}
\big( \ugr{(2,2) } {}/  \ugr{1} {} \big) \rtimes    \langle \tau \rangle   \,  \simeq  \, \su{(2,2) } / \set{\pm \mathbf{1}, \pm \ii \mathbf{1 } }  \rtimes    \langle \tau \rangle  \mapright{\simeq}  \so{2,4}{*}/\set{\pm \mathbf{1}}\\
\su{(2,2) } / \set{\pm \mathbf{1}, \pm \ii \mathbf{1 } } \,  \mapright{\simeq}   \so{2,4}{+} /\set{\pm\mathbf{1}}.
\end{array}
\]
 \end{rmk}
 \begin{rmk}\label{AltDescr}
 The (symplectic) basis $\set{a_1,a_2,b_1,b_2}$ of $V$ can be used to define   a  linear isomorphism
\begin{equation}\label{eqn:DET} 
\det :\mathsf{\Lambda}^4 V_\C \mapright{\sim} \C,\qquad a_1\wedge a_2\wedge b_1\wedge b_2 \mapsto 1.
\end{equation}
We use this  to obtain a   $\C$-antilinear involution   $t$ of $\mathsf{\Lambda}_\C  ^2 V$ as follows:
\[
t:\mathsf{\Lambda}^2  V_\C \mapright{\sim}\mathsf{\Lambda}^2  V_\C  \;\; \text{ such that } \;\;
 g(u,v)= -\det (t(v)\wedge u).
\]
By    definition of $t$ one has $t(f_i)= f_i$, $i=1,\dots,4$ and $t(f_5)=  f_6 $, $t(f_6)=   f_5 $ and since $t$ is $\C$ anti-linear 
we see that $t$ preserves not only the first 4 basis vectors $g_i$  of $T$ but also the last two $g_5,g_6$. So the  $12$-dimensional real vector space $\mathsf{\Lambda}^2  V_\C$ splits into  two real $6$-dimensional $t$-eigenspaces, namely  $T_\R=T\otimes\R$ for eigenvalue $1$ and $ \ii T_\R$ (for eigenvalue $ -1$) respectively:
\begin{equation}\label{eqn:splitwedge}
\mathsf{\Lambda}^2 V_\C= T_\R \perp \ii T_\R.
\end{equation}
 \end{rmk}
 
 \subsection{Congruence Subgroups}

The quotient $\su{(2,2) ;R}/\su{(2,2) ;R}(\omega)$ acts naturally on $(R/(\omega R))^4=\bF_2^4$. Since the  hermitian form $\ii f$ in the basis $\set{a_1,a_2,b_1,b_2}$  descends to this $\bF_2$--vector space to give a symplectic form, we then get an isomorphism
\[
\su{(2,2) ;R}/\su{(2,2) ;R}(\omega)\mapright{\sim} \smpl{4;\bF_2}.
\]
By \eqref{eqn:splitwedge} we have  $T  \perp \ii T =\lwedge ^2  V_R$.  Lemma~\ref{ClassicalIsom}, states that the group  $\su{(2,2) ;R}$ acts on $T $ and so the subgroup $\su{(2,2) ;R}(\omega)$ acts on  $(1+\germ i)^2\lwedge^2  V_R= 2\lwedge^2  V_R$. It preserves the  sublattice $2T \subset T $. 
It follows that under the isomorphism of Lemma~\ref{ClassicalIsom} one gets an identification
\begin{equation} \label{eqn:CongruIdent}
\hspace{-1 em} \smpl{4;\bF_2}\mapleft{\sim}\!\! \su{(2,2) ;R}/\su{(2,2) ;R} (\omega)\mapright{\sim} \!\!\so{T }{+} /\so{T }{+} (2).
\end{equation}
Since the involution $\tilde\tau\in  \ogr{2,4}{+}$ (see \eqref{eqn:TauTilde}) obviously does not belong to the congruence $2$ subgroup, we may define extensions as follows:
\[
\so{T }{*}(2) := \so{T}{+}(2) \ltimes \langle \tilde\tau\rangle,\quad
\ugr{(2,2) ;R}{*}(\omega) := \ugr{(2,2) ;R} {}(\omega) \ltimes \langle \tau\rangle.
\]
In particular we have
\begin{equation} \label{eqn:CongruIdent2}
\hspace{-1 em} \smpl{4;\bF_2} \mapleft{\sim}\!\! \ugr{(2,2) ;R}{*}/\ugr{(2,2) ;R}{*}(\omega) \mapright{\sim} \!\!\so{T }{*} /\so{T }{*}(2).
\end{equation}

\begin{rmks} \label{roots} 1. See  also  \cite[\S~1.5]{Mats}, where the result is shown by brute force. 
To compare, we need a dictionary. The group  $\Gamma\cA$ from loc. cit. is  our   $\so{T}{*}/\pm 1$.  The group $\Gamma\cA(2)$   is the congruence 2 subgroup which is equal to   $\ogr{T }{+} (2)/\pm 1$. It  lacks the involution $\tilde\tau$ and hence has  index 2  in the extended group $\ogr{T }{+} (2)\ltimes \langle \tilde\tau\rangle$ modulo its center. This explains why $\Gamma\cA/\Gamma\cA(2)\simeq \smpl{4;\bF_2}\times \Z/2\Z$.
\\
2. As in  \cite[\S3]{kondo} it can be shown that  the restriction homomorphism $\ogr{T(2)}{}\to \ogr{q_{T(2)}}{}$ is surjective with kernel the congruence subgroup $\ogr{T(2)}{}(2)$. The orbits of  $\ogr{q_{T(2)}}{}$ acting on  $T(2)^*/T(2)\simeq \bF_2^6$ have been described explicitly in loc. cit., using coordinates induced by the standard basis for $U(2)\perp U(2) \perp A_1\perp A_1$. The form $q_{T(2)}$ is $\Z/2\Z$-valued on the sublattice $F=\sett{a=(a_1,\dots,a_6)\in \bF_2^6}{a_5+a_6=0}$ and $b_{T(2)}$ restricts to zero on $F^0=\set{0,\kappa=(1,1,1,1,0,0)}$. Hence $b_{T(2)}$ induces a symplectic form on $F/F^0\simeq \bF_2^4$ (this explains anew that $\Gamma\cA/\Gamma\cA(2)\simeq \smpl{4;\bF_2}\times \Z/2\Z$). The orbits are now as follows:
\begin{enumerate}
\item two orbits of length $1$: $0$ and $\kappa$;
\item $\set{ a\not= 0, q_{T(2)}=0}$, the orbit (of length $15$) of $(1,0,0,0,0,0)$;
\item   $\set{ a\not= \kappa ,  q_{T(2)}=1}$, the orbit (of length $15$) of $(1,1,0,0,0,0)$;
\item $\set{ a,    q_{T(2)}=\half }$, the orbit (of length $12$) of $(1,1,0,0,1,0)$;  it splits into two equal orbits  under $\so{T(2)}{*}/\so{T(2)}{*}(2) $ (the involution exchanging the last two  coordinates  act  as the identity in this quotient);
\item $\set{ a,   q_{T(2)}=-\half }$, the orbit (of length $20$) of $(0,0,0,0,1,0)$; it splits also into two equal orbits  under $\so{T(2)}{*}/\so{T(2)}{*}(2) $;
\end{enumerate}

 \end{rmks}

 \subsection{The Corresponding Symmetric Domains}\label{ssec:Domains}
Recall \cite[\S~1.1]{Mats} , that the symmetric domain  associated to the group $\ugr{n,n}{}$, $n\ge 1$ is the $n^2$--dimensional domain\footnote{Recall: for a matrix $W$, we abbreviate $W^*=\Tr \bar W$.}
\[
\mathbf{H}_n:= \left\{ W\in M_{n\times n}(\C) \mid  \frac{1}{2\ii}(W-W^*)>0\right\} \simeq \ugr{n,n}{}/\ugr{n}{}\times\ugr{n}{}.
\]
Indeed, writing
\[
\gamma=\begin{pmatrix}
A & B\\
C & D
\end{pmatrix}\in \ugr{n,n}{}, \quad A,B,C,D \in M_{n\times n}(\C)
\]
the action is given by $\gamma(W)= (AW+B)(CW+D)^{-1}$. The full automorphism group of $\mathbf{H}_n$ is  the semi-direct product $[ \ugr{n,n}{}/\ugr{1}{}] \rtimes \langle \tau \rangle $ where 
$\tau$ is the involution given by $\tau(W)=W^{-1}$.
Since  $ \tau\comp \gamma\comp \tau= \overline{\gamma} $ this indeed corresponds to complex conjugation on $\su{n,n}$ (see Remark~\ref{ExtIso}).

The symmetric domain associated to a  bilinear  form $b$ of signature $(2,n)$, $n\ge 2$ is the $n$-dimensional connected \textbf{bounded domain of type IV} 
\[
\mathbf{D}_{n}:= \sett{z=[(z_1:\dots:z_{n+2})]\in \bP^{n+1}}{ \Tr z  b z=0,\, z ^*  b z >0, \im(z_3/z_1)>0}.
\]
Without the second defining inequality the resulting domain is no longer connected; the subgroup $\ogr{2,n}{+}$  preserves each connected component. 

This domain  parametrizes polarized weight $2$  Hodge structures $(T,b)$ with Hodge numbers $(1,n,1)$.  This can be seen as follows. The subspace $T^{2,0}\subset T\otimes\C$ is a line in $T\otimes\C$, i.e. a point  $z\in\PP(T)$.  The polarizing form $b$ is a form of signature $(2,n)$ and the two Riemann conditions translate into $\Tr z b z=0$ and $z^* b z>0$. These two relations determine an open subset $D(T) \subset \PP(T)$ and the moduli space of  such  Hodge structures is thus $D(T) /\ogr{T}{}$. Now  $D(T)$ has two components, one of  which is (isomorphic to)  $\mathbf{D}_n$;
both   $ \so{T}{+}$ and the involution $\tilde \tau$ preserve  the components, $\so{T}{+}$ has index $4$  in $\ogr{T}{}$ and hence 
\[
\ogr{T}{}= \underbrace{\so{T}{+ }\cup \tilde\tau \so{T}{+ } }_{\so{T}{*}}\cup \underbrace{\sigma \so{T}{+ }\cup\sigma\tilde\tau  \so{T}{+ }}_{\sigma\so{T}{* } }
\]
 where  $\sigma$ permutes the two components of $D(T)$: our  moduli space  can be written as the orbit  space  
 \[
 \mathbf{D}_n/\so{T}{*}.
 \]
\begin{rmk} \label{ExchCompo} Let us specify this to the case which interests us most, $n=4$. Then, $\mathbf{D}_4= \so{2,4}{+}  /\left ((\ogr{2}{}\times \ogr{4}{} )\cap \so{6}{}  \right) $ and its automorphism group is $\ogr{2,4}{+}/\langle -\mathbf 1\rangle $ 
(acting as group of projectivities on the projective space $\bP^5$ preserving the quadric in which $\mathbf{D}_4$ is naturally sitting). 

The element $c:=\text{diag}(1,1,-1,-1,1,1)\in \so{2,4}{}$  preserves the lattice $T$ and exchanges the two components; every element of  $ \so{2,4}{}$ can be written as $cg=g' c$ with $g, g' \in\so{2,4}{+}$. The element $a:=\text{diag}(1,1,1,1,U)$, $U=\begin{pmatrix} 0&1\\1&0 \end{pmatrix}$ has determinant $-1$ and every element in $\ogr{2,4}{}$ can be written as a product $acg=cag$ with $g\in \so{2,4}{+}$.
\end{rmk}

\begin{prop}[\protect{ \cite[\S1.1]{Mats}}] There is a classical isomorphism between the two domains $\mathbf{H}_2\mapright{\sim} \mathbf{D}_4$ which is equivariant with respect to the isomorphism
$$
[\ugr{(2,2) }{} /\ugr{1}{} ]  \rtimes    \langle\tau \rangle   \mapright{\simeq}   \so{2,4}{*}/\set{\pm \mathbf{1}}$$ 
of Remark~\ref{ExtIso}.  
\end{prop}
\noindent
\textit{Sketch of proof:} We describe the isomorphism briefly as follows.  To the matrix $W\in\mathbf{H}_2$ one associates the $2$--plane  in $ \C^4$ spanned by the rows of the matrix $(W | \mathbf{1}_2)$. This sends $\mathbf{H}_2$ isomorphically to an open subset in the Grasmannian of $2$--planes in $\C^4$ which by the Pl\"ucker embedding gets identified   with an open subset of the Pl\"ucker quadric in $\bP^5$. This open subset is the type IV domain $\mathbf{D}_4$.  \qed
\endproof
\begin{rmk}\label{CharactH2}
Recall we started  \S~\ref{ssec:isoms} with a  pair $(V,f)$, with $V=\Q(\ii)^4$ and $f$ a non-degenerate skew-hermitian form on $V$.
The above assignment $W\mapsto $  {the plane $P=P_W\subset V$ spanned by the rows of the matrix $(W | \mathbf{1}_2)$}  is  such that the hermitian form $\ii f| P$ is positive.  
In other words, the domain  $\mathbf{H}_2$ parametrizes the   complex 2--planes $P$  in $V_\C=V\otimes\C=\C^4$ such  $(\ii\cdot f)| P>0$. 
\end{rmk}

\subsection{Orbits in The Associated  Lattice.}\label{ssec:Orbits}

Recall that $T:=U^2\oplus\qf {-1}^2$ is the lattice on which the group $\ogr{T}{*}$ acts by isometries.  To study the orbits of vectors in $T$ we use  the results from \cite{wall}. We summarize these for this example. Recall that a primitive vector $x$ in a lattice is called \textbf{characteristic} if $\qf{x,y}\equiv \qf{y,y} \bmod 2$ for all vectors $y$ in the lattice. Other vectors are called \textbf{ordinary}. In an even lattice all primitive vectors are characteristic.  In the standard basis $ \set{e_1,e_2,e_3,e_4,e_5,e_6}$ for $T$ we let $\mathbf{x}=(x_1,x_2,x_3,x_4,x_5,x_6)$ be the coordinates. Then $\mathbf{x}$ is characteristic if and only if  $x_1,\dots,x_4$ are even and $x_5,x_6$ are odd.  The \textbf{type} of a primitive lattice vector
is said to be $0$ for ordinary vectors and $1$ for characteristic vectors. Wall's result  formulated for $T$ states that two primitive vectors of the same norm squared and of the same type are in the same $\ogr{T}{}$--orbit. So we have:
 
\begin{prop} $T=U^2\oplus\qf {-1}^2$ is a unimodular odd indefinite lattice of signature $(2,4)$, isometric to $\qf {1}^2\perp \qf{-1}^4$.  Let $\mathbf{x}\in T$ be primitive with $\qf{\mathbf{x},\mathbf{x}}=-(2k+1)$, respectively $-2k$, $k>0$. 
\begin{itemize}
\item In the first case $\mathbf{x}$ is always non-characteristic and $\ogr{T}{}$--equivalent to $(1,-k,0,0,1,0)$. 
\item In the second case, a   vector  $\mathbf{x}$ is $\ogr{T}{}$--equivalent to $(2,\half (-k+1),0,0,1,1)$ if  characteristic (and then $k\equiv 1 \bmod 4$)  and to
 $(1,-k,0,0,0,0)$ if not.
\end{itemize}
\end{prop}
\begin{rmk} \label{OtherGroups}
1) From the description of the subgroups $\so{T}{}$ and $\so{T}{+}$  in Remark~\ref{ExchCompo}, we see that  the   ``extra'' isometries $c$ and $a$ do not change  the two  typical vectors $(1,-k,0,0,0,0)$, $(2,\half (-k+1),0,0,1,1)$ while   $a$ replaces  $(1,-k,0,0,1,0)$ by  $(1,-k,0,0,0,1)$. This can be counteracted upon applying   the map $\text{diag}(-1,-1,-1,1,U) \in \so{T}{+}$. In other words, the preceding Proposition remains true for orbits under the two subgroups $\so{T}{}$ and $\so{T}{+}$.

\noindent
2) Suppose that $-d=\qf{\mathbf{x},\mathbf{x}}$  is a negative   even number. It follows quite easily that in the non-characteristic case  $\mathbf{x}^\perp $ is isometric to $\qf{d}\perp U\perp \qf{-1}\perp \qf{-1}$. In the characteristic case this is subtler. For instance, if $d=-2k$ and $k$ is a sum of two squares, say $k=u^2+v^2$, the vector  $\mathbf{x}$ is in the orbit of $(0,0,0,0,u+v,u-v)$ and so $\mathbf{x}^\perp\simeq  U\perp U\perp \qf{-d} $. This is the case if $k=a^2b$ with $b$ square free and $b\equiv 1 \bmod 4$. However, in the general situation the answer is more complicated. 
The situation over the rational numbers  is easier to explain. For later reference we introduce
\[
\Delta(\mathbf{x} ):=    -\half\qf{\mathbf{x} ,\mathbf{x} } = \half (x_5^2+x_6^2- 2(x_1x_2+ x_3x_4)) >0.
\]
Then, completing the square, one finds:
\begin{equation}\label{eqn:OverQ}
\mathbf{x}^\perp \sim_\Q U\perp \qf{-2}^2\perp \qf{2\Delta(\mathbf{x} )}.
\end{equation}
\end{rmk}

%
%
 
\begin{corr} \label{OneOrbit}
Consider  the set $Y\subset T$ of vectors of the form
 \[
\mathbf{y}=(2y_1,2y_2,2y_3,2y_4,y_5+y_6,y_5-y_6)\in T , y_i\in \Z,\, \gcd(y_1,\dots,y_6)=1.
\]
We have
\[
\Delta(\mathbf{y} ) =   y_5^2+y_6^2- 4(y_1y_2+ y_3y_4).
\]
If $y_5\not \equiv y_6\bmod 2$, the vector is a characteristic  primitive vector, $\Delta\equiv 1  \bmod 4$ and $\mathbf{y}$ is in the orbit of $(2, \frac 12(1- \Delta),0,0,1,1)$

If $y_5\equiv y_6\bmod 2$ the vector  $\half\mathbf{y} \in T$ is primitive and non-characteristic and either $\Delta\equiv 0\bmod 4$ and $\half\mathbf{y}$ is in the orbit of $(1, -\frac 14 \Delta,0,0,0,0)$, or $\Delta\equiv 2\bmod 4$ and $\half\mathbf{y}$ is in the orbit of $(1, \frac 14(2- \Delta),0,0,1,0)$

 Hence two vectors in $Y$  with the same $\Delta$--invariant are in the same $\ogr{T}{}$--orbit.
 Conversely, if the $\Delta$--invariants are different the vectors are in different orbits.
\end{corr}   
 
To be able to make a comparison between \cite{Her} and \cite{Mats} we have to realize  that
$T(2)$ is the transcendental lattice of the K3 surface in \cite{Mats} while the lattice which plays a role in \cite{Her} is the lattice $T(-1)$. 
The special vectors $\mathbf{y}\in Y$ related to Hermann's paper are not necessarily primitive. A primitive vector $\mathbf{y}^*$ in the line $\Z\cdot \mathbf{y}$  is called a \emph{primitive representative} for $\mathbf{y}$. This vector will be considered   as a vector of $T(2)$. This makes the transition between the two papers possible.

\begin{exmple} 
\label{warning}
As a {\bf warning}, we should point out
that it might happen that primitive vectors $ \mathbf{y}^*$ with the same norm squared in $T(2)$ correspond to \emph{different} $\Delta(\mathbf{y})$. For example $\mathbf{y}=(2,-2,0,0,0,0)$ and $(0,0,0,0,1,1)$  correspond to $(1,-1,0,0,0,0)$, respectively $(0,0,0,0,1,1)$. Both vectors in $T(2)$ have norm squared $-4$ while the first (with $y_1=1,y_2=-1$) has $\Delta= 4$  and the second  has $\Delta=1$, since $y_5=1,y_6=0$. From the above it follows that the  two are \emph{not} in the same orbit under the orthogonal group of $T(2)$.
\end{exmple}

Observe now that divisors in our moduli space   are cut out by hyperplanes in  $\bP(T\otimes\C)$ orthogonal to elements  $t\in T$ and any multiple of $t$  determines  the same divisor. We get therefore all possible divisors by restricting ourselves to the set $Y\subset T$. Accordingly we use the subgroup of $\so{T}{*}$ preserving the set $Y$ of vectors of this  form.   
Then  it is natural to consider the  basis $\set{2g_1,2g_2,2g_3,2g_4,g_5+g_6,g_5-g_6}$ so that the new coordinates of $\mathbf{y}$ become $(y_1,y_2,y_3,y_4,y_5, y_6)$. We may identify this vector with $\mathbf{y}^*$.

\begin{rmk}\label{OrbitUnderCongSubGroup} Suppose  $\mathbf{y}\in Y$  as in Corr.~\ref{OneOrbit}. Set  $\Delta(\mathbf{y} )=\Delta$. 
Let $n(\Delta)$ the number of different $\so{T}{*} (2)$--orbits in a given $\so{T}{*} $-orbit  for  $ \mathbf{y}\in T$ when $\Delta \equiv 0 \bmod 4$, respectively $\half\mathbf{y}$ else.
Using  Corr.~\ref{OneOrbit} and 
Rem.~\ref{roots}.2 one sees  the following:
\begin{itemize} 
\item If $\Delta \equiv 0 \bmod 4$ then $n(\Delta)=15$;
\item If $\Delta  \equiv 2 \bmod 8$ then $n(\Delta)=10$;
\item If $\Delta  \equiv 6 \bmod 8$  then $n(\Delta)=6$;
\item If $\Delta  \equiv 1 \bmod 4$  then $n(\Delta)=1$.
\end{itemize}
Indeed, for instance  a vector for which   $\Delta\equiv 0 \bmod 4$ is  in the $\so{T}{*} $-orbit of $(1,-\frac{\Delta}{4},0,0,0,0)$ which corresponds to $(1,0,0,0,0,0)$ or $(1,1,0,0,0,0)$ in $\bF_2^6$ according to whether $\Delta\equiv 0\bmod 8$ or $\Delta\equiv 4\bmod 8$. Both of these have orbitsize $15$ under $\ogr{q_{T(2)}}{}$. If $\Delta  \equiv 2,6 \bmod 8$, in applying  Rem.~\ref{roots} 2., one has to take care of the extra involution explaining why the number of orbits $n(\Delta)$ is half the orbitsize under $\ogr{q_{T(2)}}{}$.

In fact, this result is completely equivalent to \cite[Prop. 2]{Her} which is stated below (Prop.~\ref{SpecialDivisors}).
\end{rmk}

\section{A Moduli Interpretation: the Abelian Varieties Side}\label{sec:ModAV}

\subsection{Special Abelian Varieties}\label{ssec:SAV}
We say that an even  dimensional  polarized Abelian variety $(A,E)$ is of \textbf{$K$-Weil type} if $\End_\Q (A)$ contains an imaginary quadratic field $K=\Q(\alpha) \subset \C$ such that the action of $\alpha$ on the tangent space of $A$  at $0$  has  half of its eigenvalues equal to $\alpha$ and half of  its  eigenvalues equal to $\bar\alpha$ (note that this does not depend on the embedding $K\into \C$). 
We say that \textbf{$\alpha$ has type $(k,k)$}, where $2k=\dim A$.
Moreover, we want that $E(\alpha x,\alpha y)=|\alpha|^2 E(x,y)$.  
If  $\alpha=\ii$ this means that $A$ admits  an automorphism $M$ with $M^2=-1$ which preserves the polarization. Equivalently, $R\subset \End(A,E)$.

As is well known (cf. \cite[\S~1.2]{Her}, \cite[\S 10]{vG}), the   symmetric domain $\mathbf{H}_2$ parametrizes  such  Abelian $4$--folds of $\Q(\ii)$--Weil type.  
We recall briefly how this can be seen. Consider the lattice $V_R =R^4$ equipped with the skew form $\mathbf{J}$ (see \eqref{eqn:deff}). The complex vector space $V= \C^4$, considered as a real vector space, contains $V_R$ as a lattice and $\mathbf{J}$ is a unimodular integral form on it. Weight  $-1$ Hodge structures on $V$ polarized by this form are given by complex structures $J$ that preserve the form. They correspond to  principally polarized Abelian $4$-folds $A$, and if $J$  commutes with multiplication by $\ii$ the Abelian variety $A$  admits an order 4 automorphism $M$ of type $(2,2)$ which preserves the polarization. The converse is also true. 

To get the link with the domain $\mathbf{H}_2$, 
recall from Remark~\ref{CharactH2} that points in $\mathbf{H}_2$ correspond precisely to complex $2$-planes $P$ in $V$ for which $(\ii\cdot f)|P>0$.  The  direct sum splitting of  $V= P\perp P^\perp$  can be used to define a  complex structure  $J$ on the $8$-dimensional real vector space $V$  as desired by imposing  $J|P=\ii\mathbf{1}$, $J|P^\perp=-\ii  \mathbf{1}$. This complex structure commutes with multiplication by $\ii$ on $V$ and preserves $\ii f$ (since this is a hermitian form).   The embedding \eqref{eqn:H2andSiegel}  identifies the image with those $\tau\in\eh_4$ that form the  fixed locus  of  the  order $4$ automorphism formed from $\mathbf{J}$ (see \eqref{eqn:deff}
\[
\begin{pmatrix} \mathbf{J} & \mathbf{0}_4\\
\mathbf{0}_4 &  \mathbf{J}
\end{pmatrix}
\in \smpl{4;\Z}.
\]
This automorphism corresponds to multiplication with $\ii$ on the Abelian $4$-fold.

 The discrete group  $\ugr{(2,2) ;R}{*}$ acts naturally on $\mathbf{H}_2$. It sends an Abelian variety of the given type to an isomorphic one. This is clear for $\ugr{(2,2) ;R}{ }$.  
Regarding $\tau$, by   \cite[\S~1.2]{Her} the embedding \eqref{eqn:H2andSiegel} is equivariant with respect to it; indeed, it acts as an integral symplectic matrix on $\eh_4$ and hence also $\tau$ permutes isomorphic Abelian varieties. Conversely, since $\ugr{(2,2) }{*}$ modulo its center is the full group of isomorphisms of $\mathbf{H}_2$ it follows that two isomorphic Abelian varieties with multiplication by $R$ are  in the same $\ugr{(2,2) ;R}{*}$--orbit.\footnote{If we would consider such Abelian varieties up to \emph{isogeny} we would classify the isomorphism classes of  Abelian 4-folds of $\Q(\ii)$--Weil type.}
So  the quotient   
 \[
 { \mathbf{M}}:= \mathbf{H}_2/ \ugr{(2,2) ;R} {*} 
 \]
  is  the moduli space of principally polarized Abelian fourfolds with multiplication by  $R$.

\subsection{Relation With Special Weight $2$ Hodge Structures}\label{sec:Tconstruct}

Consider the Hodge structures parametrized by the domain $\mathbf{D}_4$ introduced  in \S~\ref{ssec:Domains}. 
The construction of \S~\ref{ssec:SAV} relates  such Hodge structures to  polarized Abelian  $4$-folds $A$ with multiplication by $R$.  
Indeed,  $V=H_1(A;\R)$  underlies an integral  polarized Hodge structure of weight $-1$ and rank $8$ admitting an extra automorphism $M$ of order $4$ induced by $\ii\in\End(A)$. Giving a  polarized integral Hodge structure on $V$ of weight $1$ is the same as giving a complex structure $J$ preserving the polarization; moreover,  $M$ and $J$ commute. According to \cite[\S~3]{Lom} this can now be rephrased as follows.
 Since $V=H_1(A;\R)$  underlies a rational  polarized Hodge structure of weight $-1$ and rank $8$, the second cohomology  $ H^2(A)= \mathsf{\Lambda}^2 H_1(A)^*$, inherits a polarized Hodge structure of weight $2$ and rank $28$. We view $(V,J)$   as a $4$-dimensional complex vector space
 and hence we get a complex subspace $\mathsf{\Lambda}^2_\C V^* \subset H^2(A;\R)$ of complex dimension $6$ and hence a real Hodge structure of dimension $12$. In fact it can be seen to be rational. Recall from Remark~\ref{AltDescr} that there is a further $\C$--anti linear involution $t$  on $\mathsf{\Lambda}^2_\C V^*$. Its  $(+1)$--eigenspace $T(A)$ has dimension $6$  and gives  a polarized Hodge substructure of $H^2(A) $ of weight $2$ and Hodge numbers $(1,4,1)$ as desired. 
So, this   construction explains  the isomorphism 
\[
 \mathbf{M}:=  \mathbf{H}_2/\ugr{(2,2)  ;R}{*}   \mapright{\simeq}   \mathbf{D}_4/\so{T}{*}
\]
Hodge theoretically as the the one  induced by $ A \mapsto T(A) $.
 
 If instead we divide out by the congruence subgroup $\ugr{(2,2) ;R}{*}(\omega) $ the quotient  $\mathbf{M}^*:=  \mathbf{H}_2/\ugr{(2,2) ;R}{*}(\omega)$ under  the natural morphism $\mathbf{M}^* \to \mathbf{M}$ is    Galois  over of $\mathbf{M}$ with group $\smpl{4;\bF_2}$:
  \[
\begin{diagram}
 \mathbf{M}^*  & \mapright{\simeq} & \mathbf{D}_4/\so{T}{*}(2)   \\
\darrow{}{p}{3ex} &&  \darrow{}{}{3ex} \\
\mathbf{M}&  \mapright{\simeq}&  \mathbf{D}_4/\so{T}{*}.
\end{diagram}
\]

\subsection{Hypersurfaces in the Moduli Spaces}\label{sec:Orbs}

Any  line $[a]\in \bP(T\otimes\Q )$ defines the divisor $D_{[a]}=\sett{x\in \mathbf{D}_4}{\langle a, x\rangle =0}$  inside the domain    $\mathbf{D}_4$.  As explained in \S~\ref{ssec:Orbits}, we only consider representatives $a\in T$ which belong to the set $Y\subset T$ whose  coordinates with respect to the basis $\set{g_1,\dots , g_6}$ 
are of the form:
\begin{equation}\label{eqn:SpecialY}
\mathbf{y}=(2y_1,2y_2,2y_3,2y_4,y_5+y_6,y_5-y_6) \in T.  
\end{equation}
The corresponding divisor $D_{[\mathbf{y}]} $ inside  $ \mathbf{H}_2$ can be described  by means of  the skew symmetric matrix
\begin{equation*} 
M(\mathbf{y})  :=  \begin{pmatrix}  0 & - y_2 &   \half(y_5-\ii y_6)   & -  y_4\\
 y_2& 0 & - y_3  & -  \half (y_5+\ii y_6) \\
 - \half (y_5-\ii y_6)   &  y_3 & 0 & - y_1\\
 y_4 &    \half (y_5+\ii y_6)  &  y_1 &0 \end{pmatrix}.
\end{equation*}
Indeed, we have \cite[p. 119]{Her}
\[
D_{[\mathbf{y}]}  := \left\{W\in  \mathbf{H}_2 \mid 
\begin{pmatrix} \Tr W & \mathbf{1}_2\end{pmatrix} M(\mathbf{y}) \begin{pmatrix}  W \\  \mathbf{1}_2\end{pmatrix}=
\begin{pmatrix}\mathbf{0}_2 \end{pmatrix}
\right\}.
\]
Then $A\in \su{(2,2) }$ acts by sending $M= M(\mathbf{y})$ to 
\[
A[M]:= \Tr A M A= M(\mathbf{z}), \qquad \mathbf{z}  = \mathsf{\Lambda}^2 A(\mathbf{y} ).
\]
For a   skew symmetric matrix $M$ with coefficients  in any field $K$  the determinant is always a square in the field and any   root is called a pfaffian  of $M $ and denoted by $\text{\rm Pf}(M)$. If  $K\subset \R$ we take the \emph{positive root} and call it \textbf{the pfaffian}. By Corol.~\ref{OneOrbit} and Remark~\ref{OtherGroups}, 1) we have
\begin{prop} Given a positive integer $\Delta$, there is precisely one $\so{T}{+}$  orbit of primitive vectors $ \mathbf{y}\in Y$ 
for which  $\Delta=-\half \qf{\mathbf{y},\mathbf{y}}$.
All  such vectors   $\mathbf{y}$ are  $\su{(2,2) }$-equivalent  and   the   corresponding  pfaffians {\rm Pf}$(M_{\mathbf{y}})$ are all equal. 

Moreover, such  divisors $D_{[\mathbf{y}]}$ define  the same  irreducible divisor $D_\Delta$ in the moduli space $\mathbf{M}$.
\end{prop}

\begin{rmq} The image of such a  divisor  $D_\Delta$ in the moduli space $\mathbf{M}$ can   be considered as a three dimensional modular variety  $\eh_2/\Gamma$, where $\Gamma$ is a discrete subgroup of  a certain  modular group $\smpl{2;\R}$ (depending only on    $\Delta$). For this point of view see \cite{Her}. For the  special case $\Delta=1$  see also \eqref{eqn:XQiso} in the Introduction.
\end{rmq}

 Under the congruence subgroup $\su{(2,2) ;R}(\omega)$ there are  more orbits corresponding to the fact that $D_\Delta$ may split under the cover $\mathbf{M}^*\to \mathbf{M}$. By  Remark~\ref{OrbitUnderCongSubGroup}  we have: 
  \begin{prop}[\protect{\cite[Prop. 2]{Her}}] \label{SpecialDivisors}
 Under the $\smpl{4,\bF_2}$-cover $\pi :\mathbf{M}^*\to \mathbf M$ the divisor $\pi^{-1}D_\Delta$ associated to a primitive class $a\in T^*_\Z$ with $q(a)= \Delta$ splits in $15$, $10$, $6$ or $1$ components  if $\Delta\equiv 0 \bmod 4$, $\equiv 2 \bmod 8, \equiv 6 \bmod 8$, respectively $\equiv 1 \bmod 4$.
 \end{prop}

 \section{Moduli Interpretation: Special K3 surfaces} 
 \label{sec:ModsK3}
 
\subsection{Configurations of 6 Lines in the Plane}\label{subsec:6Lines}
For this section see \cite{MSY} and in particular  Appendix A in it.

Let $\mathbf{P}=\bP^2$ and $\mathbf{P}^*$ the  dual projective space. For any integer $n\ge 4$ a configuration  of $n$-tuples of points in $\mathbf{P}$ corresponds to  a configuration    for $n$-tuples of lines in $\mathbf{P}^*$. The notion of \emph{good  position} is easy to describe on the dual space as follows.
An  $n$-tuple of lines $(\ell_1,\cdots,\ell_n)\in(\mathbf{P}^*)^n$ is called \emph{in good  position} if the corresponding curve $\ell_1\cup\cdots\cup \ell_n\subset \bP^2$ has only ordinary double points. They form a Zariski open subset 
\[
U_n=\set{(\ell_1,\cdots,\ell_n)  \text{  in good position }} \subset  (\mathbf{P}^*)^n.
\]
The  linear group $\gl {3;\C}$ acts on this space and we form the quotient
\[
\begin{array}{lcl}
X_n:&=& U_n/ \gl{3;\C}: \text{\bf configuration space of $n$ ordered lines}\\
&&\hspace{6.5em} \text{\bf   in good  position in $\bP^2$.}
\end{array}
\]
This is a Zariski open subset of $\C^{2(n-4)}$.
The  symmetric group $\germ S_n$ acts on $X_n$ and the quotient   is the configuration space of $n$  unordered  lines in general position. 
 
If $n=6$ there is an extra involution on $X_6$ induced by the correlation 
\[
\delta: \mathbf{P}\to \mathbf{P}^*,\quad x \mapsto \text{ polar of } x \text{ with respect to a nonsingular conic } C
\]
as follows: the 6 lines  $\set{\ell_1,\dots,\ell_6}$ form 2 triplets, say $\set{\ell_1,\ell_2,\ell_3}$ and $\set{\ell_4,\ell_5,\ell_6}$  each having precisely 3 intersection points. If we set $P_{ij}=\ell_i\cap\ell_j$ we thus get the triplets $\set{P_{12},P_{13},P_{23}}$ and $\set{P_{45},P_{46},P_{56}}$.  
A correlation $\delta$ is an involutive  projective transformation: it sends the line through $P$ and $Q$ to the line through $\delta(P)$ and $\delta(Q)$. In particular, three distinct  points which are the vertices of a triangle are sent to the three   sides of some (in general different)  triangle. This gives already a involution on the variety  of three ordered non-aligned points which is easily seen to be holomorphic. However, it descends as the trivial involution on the space of unordered non-aligned triples since a projectivity maps any such triple to a given one.

The above procedure for  six points gives a holomorphic involution on  $U_6$  
\[
*_C(\ell_1,\dots,\ell_6)= (\delta P_{12} ,\delta P_{13},\delta P_{23},\delta P_{45},\delta P_{46},\delta P_{56}).
\]
It descends to an involution   on $\mathbf{X} =X_6$  which  does not depend on the choice of $C$ and commutes with the action of the symmetric group $\germ S_6$. We set
\[
\mathbf{Y} =\mathbf{X}  /  \set{*}.
\]
The involution $*_C$ has as  fixed point set on $U_6$ the   6-tuples of lines all  tangent to  the conic $C$. On $X$ this gives a  non-singular divisor $XQ \subset \mathbf{X} $, the configuration of 6-uples of lines tangent to some fixed conic. This shows in particular that $*_C$ is a non-trivial involution, in contrast to what happens for triplets of points.

We need an alternative description of this involution:

 \begin{prop} \label{Cremona} The involution $*$ is,
  up to a projective transformation,
   induced by a  standard Cremona transformation with fundamental points $P_{14}, P_{25}, P_{36}$.
\end{prop}

\begin{proof}
For the proof consult  also   Fig.~\ref{fig:GoodPosition}.
The points $P_{14}, P_{25}, P_{36}$ form  a triangle $\Delta_0$. The lines $\ell_1$, $\ell_2$, $\ell_3$ form a triangle $\Delta_1$ and $\ell_4,\ell_5,\ell_6$ another triangle $\Delta_2$. In the dual plane the sides of the triangle $\Delta_0$ correspond to three non-collinear points, say $p,q,r\in \mathbf{P}^*$.
We denote the points in $\mathbf{P}^*$ corresponding to the lines $\ell_j$ by the same letter.
The configuration of the three triangles $\Delta_0,\Delta_1$, $\Delta_2$ is self-dual in the obvious sense. Note that the cubic curves $\ell_1\ell_2\ell_3=0$ and $\ell_4\ell_5\ell_6=0$ span a pencil of cubics passing through the vertices of the union of the  triangles $\Delta_1\cup \Delta_2$. The same holds for the dual configuration in $\mathbf{P}^*$.

By \cite[p. 118--119]{Do-Ort} this implies that the standard Cremona transformation $T^*$ with fundamental points  $p,q,r$ transforms these "dual" 6 vertices into a so-called associated 6-uple. To see this, let $Y$ be the $3\times 6$ matrix whose columns are the vectors of the six points $\set{\ell_1,\dots,\ell_6}$ (in some homogeneous coordinate system). The corresponding matrix $Y^*$ for the associated point set $\set{T^*\ell_1,\dots,T^*\ell_6}$ by definition satisfies $Y\Lambda \Tr Y^*=0$ for some diagonal matrix $\Lambda$. In our case we can take coordinates in such a way that $Y=(I_3 ,A)$ with $A$ invertible and after a projective transformation we may assume that  $Y^*= (I_3, -\frac{1}{\det A} A^*)$, where $A^*$ is the matrix of cofactors of $A$ so that $A\,\Tr A^*= \det(A) I _3$. This exactly means that the point set which gives $Y$ is related to the point set given by $Y^*$ by the involution $*_C$ where $C$ is the conic $x^2+y^2+z^2=0$.  See the calculations in \cite[Appendix A2]{MSY}.

%
%
%
\qed
\end{proof}

\medskip

The space $\mathbf{X}  $ can be compactified to $\bar{\mathbf{X} }  $ by adding  certain degenerate configurations to which the involution  $*$ extends and the resulting compactification   $\bar {\mathbf{X} }=\bar {\mathbf{X} }/\set{*}$  is naturally isomorphic to $\bP^4$.   The group $\germ S_6$ acts on both sides giving a commutative diagram
\[
\begin{diagram}
 \bar{\mathbf{X} }  & \mapright{\sigma} &  \bar {\mathbf{Y} } \simeq \bP^4   \\
\darrow{}{}{3ex} &&  \darrow{}{\pi}{3ex} \\
\bar {\mathbf{X} }/\germ S_6 &  \mapright{\bar\sigma}&  \bar {\mathbf{Y}}/\germ S_6.
\end{diagram}
\]

\subsubsection{Comparison with Hermann's work}

We now compare the result with \cite{Her}. We need some more details of the above construction. To start with we choose coordinates in $\mathbf{Y} $ by representing first a point in $\mathbf{Y} $ by a $3\times 6$ matrix $(x_{ij})$ (the 6 rows give the six lines) and 
let $d_{ijk}(x)$ be the minor obtained by taking columns $i,j,k$. Then for every permutation $\set{ijk \ell m n}$   of $\set{1,\dots,6}$, consider the 10  Pl\"ucker coordinates $Z_{ijk}:=d_{ijk} d_{ \ell m n}$ which one uses to embed $\mathbf{Y} $ in $\bP^9$. The Pl\"ucker relations $Z_{ijk}-Z_{ij\ell}+Z_{ijm}-Z_{ijn}=0$ then show that this embedding is a linear embedding into $\bP^4\subset \bP^9$.   
Note that the permutation group $\germ S_6$ interchanges the Pl\"ucker coordinates and the image 4-space is invariant under this action.

In \cite[\S~4]{Her} an embedding\footnote{What Hermann calls $\bar X(1+\ii)$ is in our   notation $\bar Y$ and in Matsumoto's notation $Y^*$. }  of  $\bar {\mathbf{Y} }$ into $\bP^5$ with homogeneous coordinates $(Y_0,\dots,Y_5)$ is constructed with image the hyperplane $(Y_0+\cdots+Y_5)=0$.
We may assume that the $Y_i$, $0\le i\le 4$ coincide with some of our Pl\"ucker coordinates which we write accordingly as $\set{Y_0,\dots,Y_9}$. Indeed,  Hermann's six   coordinates are permuted as the standard permutation of $\germ S_6$ on 6 letters and we can scale them in order that the hyperplane in which the image lies is given by the equation  $(Y_0+\cdots+Y_5)=0$.

One of  the   divisors in $\bar {\mathbf{X} }$ added to $\mathbf{X} $ is  needed below. It is called $X_3$ in \cite{MSY} and  has  the $20$ components $X^{ijk}_3$ of configurations of 6 lines  where precisely the  three lines $\ell_i,\ell_j,\ell_k$ meet at one point which create one  triple point. There are 12 further  points where only 2 lines meet.    We now show how to  identify $\sigma(X_3)$  with the divisor $D'_2=\pi^{-1}D_2$  from Prop.~\ref{SpecialDivisors}.  First invoke \cite[Prop. 2.10.1]{MSY} (see also Theorem~\ref{MSYTheorem} below)  which makes the transition from $\mathbf{X}$ to $\mathbf{M}$ possible. Next, from \cite[p. 122-123]{Her}, we infer  that the equation of $D'_2$ reads
\[
  \prod_{abc} (Y_a + Y_b + Y_c)=0.
\]
By  the Pl\"ucker  relations  $Y_a +Y_b +Y_c = \pm Y_d$ for some $d\in\set{0,\dots,9}$. 
Hence $ Y_a + Y_b + Y_c  =  \pm Y_d $, say $\pm Y_d=Z_{ijk}$  
and  the   zero locus of that factor corresponds
to $D_{ijk} \cdot D_{lmn} = 0$.  It follows that indeed $D'_2=X_3$. We observe that there are 20  
irreducible components  $X_3^{ijk}$ but each $Y_d$ gives two of them via the double cover $\sigma$, so we get indeed the 10  divisors of Hermann.

As a side remark,  $\bar{\mathbf{X} } \setminus \mathbf{X} $ contains further divisors, several of which parametrize K3 surfaces, namely whenever the double points in the configuration coalesce to  triple points at worst.
 
We want to stress that the definition of \emph{good} includes rather special configurations, one of which is needed below, namely the ones forming a divisor $\mathbf{X}_{\rm coll}\subset \bar {\mathbf{X} } $ corresponding to 6 lines  $\set{\ell_1,\dots,\ell_6}$ where the intersection points $P_{12}=\ell_1\cap \ell_2, P_{34}=\ell_3\cap \ell_4, P_{56}=\ell_5\cap \ell_6$ of three pairs of lines are collinear. We can identify $\sigma \mathbf{X}_{\rm coll}$ with $\pi^{-1}D_4$ as follows.
From \cite[p.~123]{Her}, we find that the equation of $D_4$ is
\[
 \prod_{ij}  Y_i -Y_j=0.
\] 
The  Pl\"ucker  relations 
yield the equations 
\[
D_{abc}D_{efg} = D_{a'b'c'}D_{e'f'g'}
\]
where in addition to
$\{a,b,c,e,f,g\} = \{a',b',c',e',f',g'\} = \{1,\hdots,6\}$ necessarily (up to commuting
the factors of the products) 
\[
\# \{a, b, c\} \cap \{a', b', c'\} = \# \{e, f, g\} \cap \{e', f', g'\} = 2.
\]
Dually looking at 6 points in $\PP^2$, we obtain the same result if three lines (spanned
by different pairs of such  points) meet in one and the same point. So indeed, we get $\sigma \mathbf{X}_{\rm coll}$. Observe that there are $\frac 1 6 15 \times 6= 15$ ways to make the intersection points collinear giving 15 components as it should.

\subsubsection{Relation to recent work of Kond\=o}

In a recent preprint \cite{kondo},
Kond\=o also studies Heegner divisors on the moduli space
of 6 lines in $\PP^2$
using Borcherd's theory of automorphic forms on bounded symmetric domains of type IV.
Specifically he singles out four divisors in \cite[\S 3]{kondo} which correspond to the cases $\Delta=1,2,4,6$ 
studied extensively in this paper.

\subsection{Double Cover Branched in $6$ Lines in Good Position}
\label{ss:intro}

We refer to \cite[Ch VIII]{fourauthors} for details of the following discussion on moduli of K3-surfaces. 
 The second cohomology  group of K3 surface  $X$ equipped with the cup product pairing is known to be isomorphic to the unimodular even lattice
 \[
 \Lambda:= U\perp U\perp U \perp E_8  \perp E_8 .
 \]
The lattice underlies a weight $2$ polarized Hodge structure with Hodge numbers $h^{2,0}=1$, $h^{1,1}=20$. The N\'eron-Severi lattice $\NS(X)$ gives a sub Hodge structure with Hodge numbers $h^{2,0}=0$, $h^{1,1}=\rho$, the Picard number. Its orthogonal complement $T(X)$, the transcendental lattice, thus also is a polarized Hodge structure with Hodge numbers $h^{2,0}=1$, $h^{1,1}=20-\rho$. The N\'eron-Severi lattice is a Tate Hodge structure, but $T(X)$ has moduli. The two Riemann bilinear relations show that these kind of Hodge structures are parametrized by a type IV domain $\mathbf{D}_{20-\rho}$ (see \S~\ref{ssec:Domains}). Conversely, given a sublattice $T\subset \Lambda$ of signature $(2,n), n<20$, the polarized Hodge structures on $T$ with Hodge numbers $h^{2,0}=1$, $h^{1,1}=n$ are parametrized by a domain $D(T)$ of type IV whose points correspond to K3 surfaces with the property that the   transcendental lattice is contained in $T$. For generic such points the transcendental lattice will be exactly $T$, but upon specialization the surfaces may acquire extra algebraic cycles which show up in $T$. In other words, the transcendental lattice of the specialization becomes strictly smaller than $T$. Note also that $\dim D(T)=n$. These surfaces, commonly called \textbf{$T$-lattice polarized K3 surfaces}, thus have $n$ moduli. It is not true that all points in $D(T)$ correspond to such K3 surfaces: one has to leave out the hyperplanes $H_\alpha$ that are perpendicular to the roots $\alpha$  in $T$; since $T$ is an even lattice, these are elements $\alpha\in T$ with $\qf{\alpha,\alpha}=-2$. 
We quote  the following result \cite[VIII, \S 22]{fourauthors} which makes this precise:
\begin{thm*} Let
\[
D^0(T):= D(T)\setminus \bigcup H_\alpha,\quad \alpha \text{ a root  in } T.
\]
The moduli space of $T$-lattice polarized K3 surfaces is the quotient of $D^0(T)$ by the group $\so{T}{*}$.
\end{thm*}

 Examples of K3 surfaces with $\rho=1$ (having  19 moduli) are the double covers of the plane branched in a generic smooth curve of degree 6 parametrized by a 19-dimensional type IV domain. If we let this curve acquire more and more singularities we get deeper and deeper into this  domain. We are especially interested in those sextics that are the union of six lines in good position (see Fig.~\ref{fig:GoodPosition}) and certain of their degenerations which were  treated in \S~\ref{subsec:6Lines}.  
\begin{figure}[h!] \includegraphics[width=\linewidth]{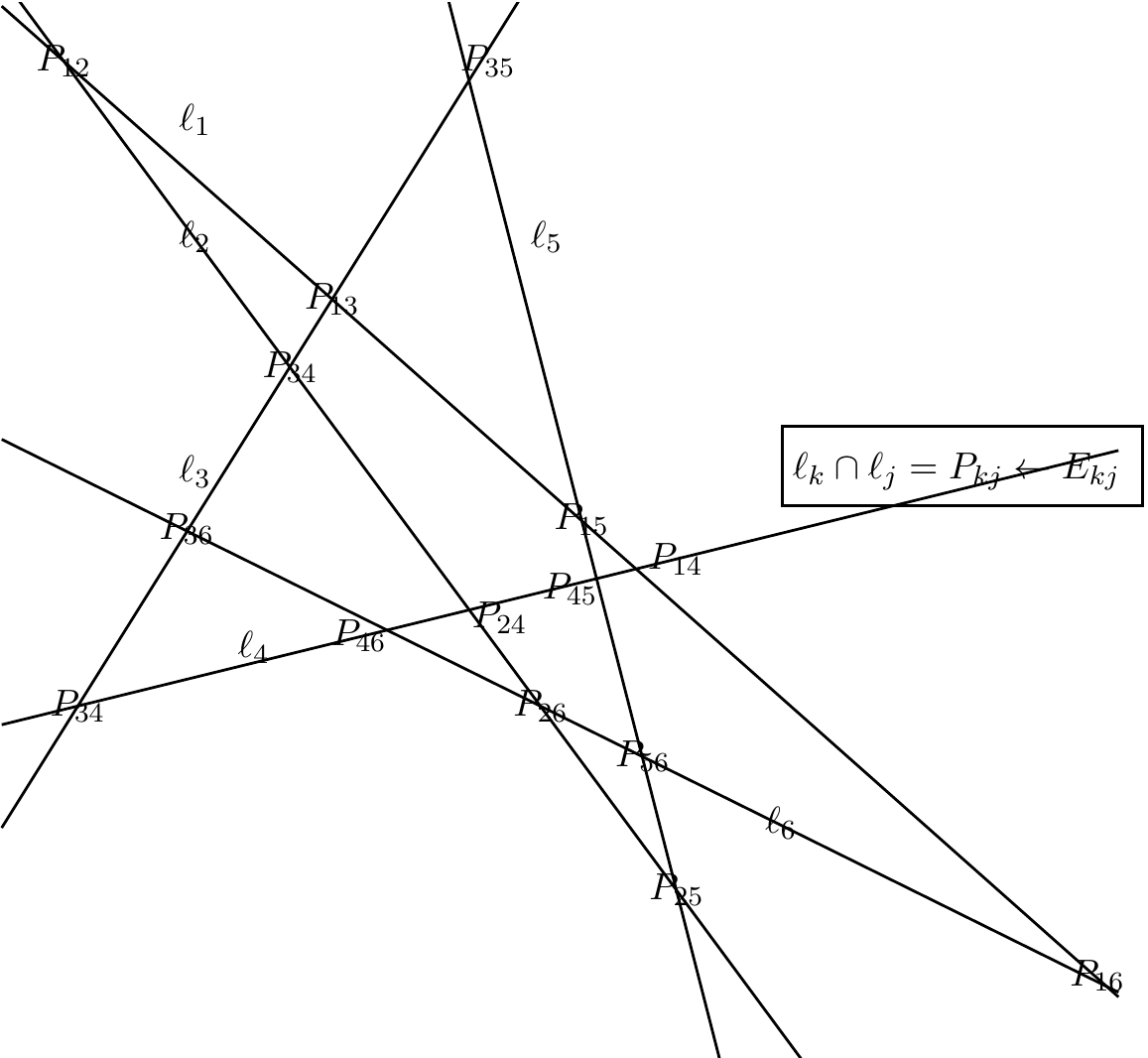}
\caption{The 6   branch lines and the 15 exceptional curves $E_{kj}$}\label{fig:GoodPosition}
\end{figure} 
\begin{prop} Let $X$ be the minimal resolution of the  double of cover of the plane given by an equation:
 \begin{eqnarray}
 \label{eq:double6}
 w^2 = \ell_1(x,y,z) \cdots \ell_6(x,y,z).
 \end{eqnarray}
 Assume that the 6 lines $\ell_i$, $i=1,\dots,6$ are in general (and, in particular, in good)  position. Then the Picard number $\rho(X)$ equals $\rho(X)=16$.
\end{prop}
\proof The 15 ordinary double points $P_{ij}\, (1\leq i<j\leq 6)$  in the configuration give 15  $E_{ij}$ disjoint exceptional divisors  on $X$; these are $(-2)$-curves, i.e.\ rational curves with self intersection $(-2)$. The generic line gives one further divisor $\ell$ with $\ell^2=2$ and which is  orthogonal to the $E_{ij}$. The divisors $\set{\ell, E_{12},\dots, E_{56}}$ thus form a sublattice $N$ of rank $16$ within the N\'eron-Severi lattice $\NS(X)$.  
So for the Picard number we have  $\rho \ge 16$. As explained above, we have $20-\rho$ moduli where $\rho$ is the Picard number of a generic member of the family. So $\rho=16$ and   hence for generic choices of lines $\ell_i$ the lattice $\NS(X)$ contains  $N$ as a sublattice of finite index. 
\qed
\endproof

We need a simple consequence of the proof. To explain it we need a few notions from lattice theory.
Recall that the  dual of  a lattice $L$ is defined by
\[
L^*:= \sett{x\in L\otimes\Q}{ \qf{x,y}\in\Z,\quad \text{for all } y\in L}
\]
and that the  discriminant group  $\delta(L)$ is the finite Abelian group $L^*/L$.
We say that $L$ is \textbf{$p$--elementary}, if this is so for $\delta(L)$, i.e.
\[
L^*/ L \cong (\Z/p\Z)^\ell, \quad \ell:=\text{\bf  length of } L \le \rank(L).
\]
From the above proof we see that 
 \[
 \langle 2 \rangle \perp  \langle -2\rangle^{15}\subset\NS(X).
 \]
and so
\begin{corr}\label{NSis2elementary} The N\'eron-Severi lattice $\NS(X)$ is $2$--elementary.
\end{corr}
 In what follows we shall   first of all determine both the N\'eron-Severi and the transcendental lattice   for such a generic K3 surface   $X$. We shall prove:
 
 \begin{thm} \label{MainTHGenericCase} For generic $X$ as above we have $\NS(X) = U\perp D_6^2\perp A_1^2$ and $T(X) = U(2)^2 \perp A_1^2=T(2)$.
\end{thm}

This  gives  an interpretation of   the  previous results  in terms of the moduli of K3 surfaces. Indeed we have  $D(T(2))=\mathbf{D}_4$ and we note that   $T(X)=T(2)$ and $T$ have the same orthogonal group. So   the above result shows that our moduli space equals 
\[
D^0(T) / \so{T}{*}= \mathbf{D}_4^0 /\so{T}{*}.
\]
Moreover, by the results of \cite{MSY} we can now identify  this moduli space with   the configuration spaces from \S~\ref{subsec:6Lines}.
\begin{thm}[\protect{\cite[Prop. 2.10.1]{MSY}}] \label{MSYTheorem}
There is a commutative diagram
\[
\begin{diagram}
\mathbf{X} /\set{*}=\mathbf{Y}    & \rarrow{\tilde p}{\simeq}{3em}&  \mathbf{D}_4^0/ \so{T}{*}(2) \subset \mathbf{M}^*  \\
\darrow{}{\pi}{3ex} &&  \darrow{}{}{3ex} \\
\mathbf{X} /\left[\set{*}\times \germ S_6\right]=\mathbf{Y}  /\germ S_6 &  \rarrow{\simeq}{p}{3em} &  \mathbf{D}_4^0 /\so{T}{*}\subset \mathbf{M}.
\end{diagram}
\]
The holomorphic maps $\tilde p$ and $p$ are biholomorphisms. They are the period maps.
\end{thm}
In other words, the moduli space  of $T$--lattice polarized K3's can be identified with  the quotient of the configuration space of unordered 6-tuples of lines in $\bP^2$ by the correlation involution $*$. Note also that the group $\germ S_6$ on the left is indeed isomorphic to the quotient
$\so{T}{*}/\so{T}{*}(2)=\smpl{4;\bF_2}$ (see \eqref{eqn:CongruIdent}).

\begin{rmk} \label{LinesAndCremona} Recall that $*$ sends the corresponding double covering K3 surface  $X$ to a K3 surface $*X$. The involution $*$ sends the 6 branch lines defining $X$ to the three branch lines defining $*X$. 
By Prop.~\ref{Cremona} there is a standard  Cremona transformation with fundamental points  the three points $P_{14}, P_{25}, P_{36}$ which induces this involution on the level of the plane and hence $X$ and $*X$ are isomorphic K3-surfaces, the isomorphism being induced by the Cremona transformation. It follows that the quasi-polarization (given by the class of a line $\ell$ on the plane) is not preserved under this isomorphism: it is sent to $2 \ell-e_{14}-e_{25}-e_{36}$ where $e_{ij}$ is the class of the exceptional curve $E_{ij}$.

 By \cite[\S 1.4.]{Mats} the involution $*$ 
corresponds to the  involution $j$ which on $T=U^2 \perp \qf{-1}^2$ fixes the first 4 basis vectors and sends the fifth to minus the sixth. 
Since $T$ as well as its orthogonal complement $S=T^\perp$ is $2$-elementary, by \cite[Theorem 3.6.2]{Nik1} the restriction $\ogr{S}{} \to \ogr{q_S}{}$ is surjective. Any lift of the image of $j$ under the homomorphism $\ogr{T}{} \to \ogr{q_T}{} $  to $S$  then can be glued together with $j$  to obtain an isometry of  the K3-lattice $\Lambda$. Such an isometry sends the period of $X$ to the period of an isomorphic K3-surface which must be $*X$ by the Torelli theorem.  
\end{rmk}

Next,  we  study what happens when the line configuraton degenerates. 
\begin{thm} \label{MainTHDegenerations}  Put
\[
X_\Delta: \, \text{the generic K3 surface  on }D_\Delta \subset \mathbf{M} 
\]
We have
\begin{enumerate}
\item   $\NS(X_2)=U \perp D_4^2\perp E_7$, $T(X_2)=U(2)^2\perp A_1$;
\item   $\NS(X_4)=U\perp D_6^2\perp A_3$, $T(X_4)=U(2)\perp\langle 4\rangle\perp A_1^2$.
\item   $\NS(X_1)=U\perp D_4\perp D_8 \perp A_3$, $T(X_1)=U(2)^2\perp \langle-4\rangle$.
\item   $\NS(X_6)=U\perp D_6^2\perp A_1 \perp A_2$, $T(X_6)=U(2)\perp A_1^2\perp \langle 6\rangle$.
\end{enumerate}
\end{thm}

 To prove this  we will make substantial use of elliptic fibrations.
 We should point out that most, if not all computations can be carried out with explicit divisor classes on the K3 surfaces;
 elliptic fibrations have the advantage of easing the lattice computations
 as well as providing geometric insights,
 since the root lattices in the above decomposition of $\NS$
 appear naturally as singular fibers of the fibration
 (conf.~for instance \cite{CD}, \cite{Nishi}).
 
 After reviewing the basics on elliptic fibrations needed,
 we will first prove Theorem \ref{MainTHGenericCase} in \ref{ss:2D6}.
 Then using lattice enhancements the three cases of Theorem \ref{MainTHDegenerations}
 will be covered in \ref{sss:3.6.1}, \ref{ss:4} and \ref{sss:3.6.9}.

 \subsection{Elliptic fibrations and the Mordell-Weil lattice}
 \label{ss:ell}

We
start by reviewing basic facts on elliptic fibrations, and in particular on the N\'eron-Severi lattice for an elliptic fibration with section following Shioda as summarized in \cite{SS}.

Let $S\to C$ be an elliptic fibration of a surface $S$, with a section $s$ and general fiber $f$. These two span a rank  $2$ sublattice $U$ of the N\'eron-Severi lattice $\NS=(\NS(S),\qf{\,,\,})$,  isomorphic to the hyperbolic plane.
We call $s$ the \emph{zero section}; it meets every singular fiber in a point which  figures  as the neutral element in a group $G_\nu$ whose structure is given in the table below.

\vspace{1ex}
 \begin{center}
\begin{tabular}{|c||c|c|c|c|c|}
\hline
Fiber type & $F_\nu$  & $e_\nu$  & $G_\nu$   & discr$(F_\nu)$  &discr. gr. \\
\hline\hline
$I_n$ & $A_{n-1}$  & $n$  &$\C^*\times\Z/n\Z$  & $ (-1)^n(n+1)$ &$\Z/n\Z$ \\
\hline
$II $ &  --  &$1$ & $\C^*$&  $1$& $\set{1}$     \\
\hline
$III $ & $A_1$ & $1$  &$\C\times \Z/2\Z$ &  $-2$& $\Z/2\Z$    \\
\hline
$IV$&  $A_2$    & $2$ &  $\C\times\Z/3\Z$ & $-3$ &  $\Z/3\Z$\\
\hline
$I^*_{2n}$&   $D_{2n+4}$ &$2n+5$   &  $\C\times(\Z/2\Z)^2$  &$4$ &$(\Z/2\Z)^2$\\
\hline
$I^*_{2n+1}$&   $D_{2n+5}$ &$2n+6$   &  $\C\times\Z/4\Z $ &$-4$ &$\Z/4\Z$\\
\hline
$II^*$ & $E_8$& $9$ &$\C$& $1$ & $\set{1}$\\
\hline
$III^*$&$E_7$& $8$ &$\C\times\Z/2\Z$& $-2$ & $\Z/2\Z$\\
\hline
$IV^*$&$E_6$&$7$&$\C\times\Z/3\Z$& $3$& $\Z/3\Z$\\
\hline
\end{tabular}
\end{center}
 In the table we   enumerate Kodaira's list of singular fibers. The components  of a singular  fiber $f_\nu$ not met by the zero section 
 define mutually orthogonal negative-definite sublattices $F_\nu$ of the N\'eron-Severi lattice, all orthogonal to 
 the hyperbolic plane $U$. The   Euler number of the fiber $f_\nu$  is abbreviated by $e_\nu$ in the table. 
 The last two entries are the discriminant  and the discriminant group  of the lattice  $F_\nu$.
 
 The  lattice
\[
T:= U\perp \bigoplus_\nu F_\nu
\]
is called the \textbf{trivial lattice} of the elliptic surface $X$. It is a sublattice of $\NS$, but \emph{not necessarily primitive}.  Its orthogonal complement (inside $\NS(X)$)
\[
L:= T^\perp_{\NS}
\]
is called the \textbf{essential lattice}. The group of sections forms the \textbf{Mordell Weil group} $E$. Its torsion part can be calculated as follows:
\[
T':=\text{primitive   closure of $T$ in } \NS; \quad T'/T \simeq E_{\rm tors}.
\]
It is one of the main  results of the theory of elliptic surfaces
that
\begin{eqnarray}
\label{eq:E}
E \cong \NS/T.
\end{eqnarray}
The most famous incarnation of this fact is often referred to as Shioda-Tate formula:
\begin{eqnarray}
\label{eq:ST}
\text{rank}(\NS) = 2 + \sum_\nu \text{rank}( T_\nu) + \text{rank}(E).
\end{eqnarray}
The main idea now is to endow $E/E_{\rm tors}$ with the structure of a positive definite lattice,
the \textbf{Mordell-Weil lattice} $\MWL=\MWL(S)$.
This can be achieved as follows.
Since $L_\Q\perp T_\Q=\NS_\Q$, the restriction to $E$
\[
\pi_E: {\NS}_{\Q}|_E \to L_\Q
\]
 of the orthogonal projection  is well-defined with kernel $E_{\rm tors}$. 
 The \textbf{height pairing} 
 on $\MWL= E/E_{\rm tors}$
 by definition is induced from the pairing on the N\'eron-Severi group:
\[
\qf{P,Q}:= -\qf{\pi_E(P),\pi_E(Q)}, \quad P, Q\in E.
\]
Note that  by definition, this pairing need not be integral.
Shioda has shown that the height pairing  can be calculated directly from the way the sections $P, Q$ meet each other,
the zero section, and in particular the singular fibers $f_\nu$.
In the sequel we will need this only for the height $\qf{P,P}$ of an individual section $P\in E$.
The only components of $f_\nu$ possibly met by a section are 
the multiplicity $1$ components not met by the zero section. 
For $I_n$,   this gives $n-1$ components that one enumerates successively, starting from the first component next to the one meeting the zero section (upto changing the orientation). For $I_n^*$,  $n>0$, there are $3$ components: the \emph{near} one (the first component next to the one meeting the zero section) and two far ones 
(for $I_0^*$ fibers the three simple non-identity components are indistinguishable). In the end, the height formula reads
\begin{equation}\label{eqn:HP}
h(P) := \qf{P,P}= 2\chi(\mathcal O_S) + 2\qf{P,s}-\sum_\nu c_\nu,
\end{equation}
where  the  local contribution $c_\nu$  for $f_\nu$ can be found in  the following table  and is determined by the component which the section  $P$ meets (numbered $i=0,1,\hdots,n-1$ as above for fibers of type $I_n$): 
\vspace{1ex}
 \begin{center} 
\begin{tabular}{|c||c|c|}
\hline
Fiber type & root lattice & $c_\nu$ \\
\hline\hline
$I_n\; (n>1)$ & $A_{n-1}$  & $\frac{i(n-i)}{n}$\\ 
\hline
$III $ & $A_1$ &    $ \half$    \\
\hline
$IV$&  $A_2$    & $ {2\over 3}$\\
\hline
$I^*_{n} \;(n\geq 0)$&   $D_{n+4}$ &$1$  (near), $1+{n\over 4}$ (far) \\
\hline
$II^*$ & $E_8$& $-$\\
\hline
$III^*$&$E_7$& ${3\over 2}$\\
\hline
$IV^*$&$E_6$&${4\over 3}$ \\
\hline
\end{tabular}
\end{center}

The following formula for the discriminant of $\NS=\NS(S)$, the N\'eron-Severi lattice can be shown to follow from the above observations:
\begin{equation}\label{eqn:discrNS}
\text{discr}(N)=\frac{(-1)^{\text{rank} E}}{|E_{\rm tors}|^2} \, \text{discr}(T)\cdot\text{discr}(\MWL).
\end{equation}

 \subsection{Generic $\NS(X)$ and the Proof of Theorem~\ref{MainTHGenericCase}}
 
 For later use, we start by computing the N\'eron-Severi lattice $\NS(X)$ and the transcendental lattice $T(X)$
 with the help of elliptic fibrations with a section, the so-called \textbf{jacobian elliptic fibrations}.
 It is a special feature of K3 surfaces that they may admit several jacobian elliptic fibrations.
For instance, we can multiply any three linear forms from to the RHS to the  LHS 
of \eqref{eq:double6} such as
\begin{eqnarray}
\label{eq:pencil}
X:\;\;\; \ell_1\cdots\ell_3 w^2 = \ell_4\cdots\ell_5.
\end{eqnarray}
Here a fibration is simply given by projection onto $\PP^1_w$;
it is the quadratic base change $v=w^2$ of a cubic pencil
with 9 base points $P_{14}, \hdots, P_{36}$ as sections.
However, this plentitude of sections (forming a Mordell-Weil lattice of rank 4)
makes the lattice computations quite complicated, so we will rather work with two other elliptic fibrations on $X$.

Note that due to the freedom in arranging the lines in \eqref{eq:pencil},
the K3 surface  $X$ admits indeed several different fibrations of the above shape.
This ambiguity will persist for all elliptic fibrations throughout this note.

\subsubsection{Standard elliptic fibration}
\label{ss:2D4}

We shall now derive an elliptic fibration on $X$ which will serve as our main object in the following.
For this purpose we specify the elliptic parameter $u$ giving the fibration by
\[
u = \ell_1/\ell_2.
\]
One easily computes the divisor of $u$ as
\[
(u) = 2 \ell_1+E_{13}+\hdots+E_{16}- 2 \ell_2-(E_{23}+\hdots+E_{26}).
\]
Both zero and pole divisor encode divisors of Kodaira type $I_0^*$,
hence the morphism
\[
u: X \to \PP^1
\]
defines an elliptic fibration on $X$ with sections $\ell_3,\hdots,\ell_6$.
Note that the exceptional divisors $E_{ij} (3\leq i<j\leq 6)$ are orthogonal to both fibers;
hence they comprise components of other fibers. 
We sketch some of these curves in the following figure:

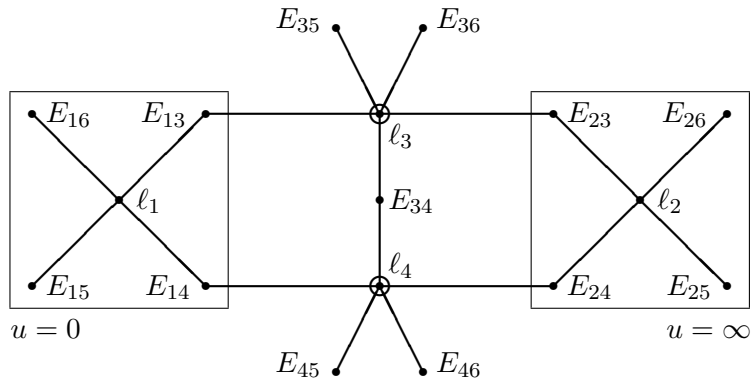
\begin{figure}[ht!]
\setlength{\unitlength}{.45in}
\begin{picture}(8,5)(-1.5,0)
\thicklines

\multiput(0,1)(2,0){5}{\circle*{.1}}

\multiput(0,3)(2,0){5}{\circle*{.1}}

\multiput(1,2)(3,0){3}{\circle*{.1}}

\multiput(3.5,0)(1,0){2}{\circle*{.1}}
\multiput(3.5,4)(1,0){2}{\circle*{.1}}

\put(0,1){\line(1,1){2}}
\put(6,1){\line(1,1){2}}

\put(0,3){\line(1,-1){2}}
\put(6,3){\line(1,-1){2}}

\put(2,1){\line(1,0){4}}
\put(2,3){\line(1,0){4}}

\put(4,1){\line(0,1){2}}

\put(3.5,0){\line(1,2){.5}}
\put(4.5,0){\line(-1,2){.5}}

\put(3.5,4){\line(1,-2){.5}}
\put(4.5,4){\line(-1,-2){.5}}

\put(2.8,0){$E_{45}$}
\put(4.65,0){$E_{46}$}

\put(2.8,4){$E_{35}$}
\put(4.65,4){$E_{36}$}

\put(.155,.9){$E_{15}$}
\put(1.3,.9){$E_{14}$}

\put(.155,2.9){$E_{16}$}
\put(1.25,2.9){$E_{13}$}

\put(6.155,.9){$E_{24}$}
\put(7.3,.9){$E_{25}$}

\put(6.155,2.9){$E_{23}$}
\put(7.25,2.9){$E_{26}$}

\put(1.2,1.9){$\ell_1$}

\put(4.1,1.9){$E_{34}$}

\put(7.2,1.9){$\ell_2$}

\put(4.1,1.15){$\ell_4$}
\put(4,1){\circle{.2}}

\put(4.1,2.65){$\ell_3$}
\put(4,3){\circle{.2}}

\thinlines
\put(-.25,0.75){\framebox(2.5,2.5){}}
\put(5.75,0.75){\framebox(2.5,2.5){}}

\put(-.25,.4){$u=0$}
\put(7.3,.4){$u=\infty$}

\end{picture}
\caption{Some sections 
and fiber components of the standard 
fibration}
\label{fig}
\end{figure}

There is an immediate sublattice $N$ of $\NS(X)$
generated by the zero section and fiber components.
Here  this amounts to 
\[
N = U\perp D_4^2\perp A_1^6.
\]
Since the rank of $N$ equals
the Picard number $\rho=16$ of $X$,
the sublattice $N$ has finite index in $\NS(X)$.
We note two consequences.
First, the $(-2)$ curves $E_{ij} (3\leq i<j\leq 6)$ generically
sit on 6 fibers of type $I_2$ 
(because otherwise there would be an additional fiber component contributing to $\NS(X)$).
For later reference, we denote the other component of the respective fiber by $E_{ij}'$;
this gives another $-2$-curve on $X$.
Secondly, we deduce from  \eqref{eq:ST}  that generically the Mordell-Weil rank is zero.
Since $I_0^*$ fibers can only accommodate torsion section of order $2$,
the given four sections give the full 2-torsion.
Alternatively, this can be computed with the height pairing as sketched in \ref{ss:ell}
or it can be derived from the  actual equations
which we give in \ref{sss:eqn-6}.
From \eqref{eqn:discrNS} we deduce that $\NS(X)$ has discriminant 
\begin{eqnarray}
\label{eq:disc}
\mbox{discr} \NS(X) =-2^{10}/2^4=-2^6.
\end{eqnarray}
Since by  Corollary~\ref{NSis2elementary}  $\NS(X)$ is 2--elementary we find that  \eqref{eq:disc} 
implies  the length of the N\'eron--Severi lattice to equal  $6$.

\subsubsection{Transcendental lattice}

We want to compute the transcendental lattice $T(X)$. Again we need some 
general facts from lattice theory which we collect at this place for the reader's convenience.

For an even non-degenerate integral lattice $(L,\qf{-,-})$ recall (see "Notation")
 the discriminant group $\delta(L)=L^*/L$ and   the   $\Q/2\Z$--valued \textbf{discriminant  form}  induced by $\qf{-,-}$ denoted  $q_L$. The importance of this invariant stems from the following result of Nikulin \cite{Nik1}:  two even lattices with the same signature and discriminant form are in the same genus, i.e. are isomorphic over the rationals.
 
Below we need a more precise result in a special situation:
\begin{prop}[\protect{\cite[Prop.~4.3.2]{Nik}}] \label{nikulin} Any indefinite 2-elementary lattice is determined up to isometry
by signature, length, and the property whether the discriminant form 
takes only integer values or not.
\end{prop}

We return to our double sextics $X$ in the generic situation.
Since $T(X)$ and $\NS(X)$ are orthogonal complements embedded primitively  into the unimodular lattice $\Lambda$,
we find by construction an isomorphism of discriminant forms
\begin{eqnarray}
\label{eq:q}
q_{T(X)} \cong -q_{\NS(X)}.
\end{eqnarray}
In particular, $T(X)$ is again 2-elementary of length $l=6$, and of signature $(2,4)$.
So we may apply  the above Prop.~\ref{nikulin}. 
To do so we need to be able to determine properties of  the discriminant form.
To decide this without going through explicit computations with divisor classes on $X$,
we switch to another elliptic fibration on $X$.

\subsubsection{Alternative elliptic fibration}
\label{ss:2D6}

In order to exhibit another elliptic fibration on $X$,
we start by identifying two perpendicular divisors of Kodaira type $I_2^*$:
\begin{eqnarray*}
D_1 & = & E_{15} + E_{16} + 2(\ell_1+E_{13}+\ell_3) + E_{35} + E_{46}'\\
D_2 & = & E_{25} + E_{26} + 2(\ell_2 + E_{24} + \ell_4) + E_{45} + E_{36}'
\end{eqnarray*}
Their linear systems induce an elliptic fibration with section induced by $\ell_6$,
since $\ell_6.D_i=1$.
In addition to the two fibers of type $I_2^*$,
there are 2 further reducible fibers with identity component
$E_{34}'$ on the one hand and $E_{56}$ on the other hand.
Rank considerations imply that their type is generically $I_2$,
so that the given fibers and the zero section generate the sublattice
$U+D_6^2+A_1^2$ of $\NS(X)$.
In fact, since ranks and discriminants agree, we find the generic equality
\begin{eqnarray}
\label{eq:2D6}
\NS(X) = U\perp D_6^2\perp A_1^2.
\end{eqnarray}
In particular, this singles out $A_1^*/A_1$ as an orthogonal summand 
of the discriminant group $\delta(\NS(X))=\NS(X)^*/\NS(X)$.
Its quadratic form thus takes non-integer values in $\frac 12\Z/2\Z$.
As $T(X)$ has the same invariants as $U(2)^2 \perp  A_1^2$, by   Prop.~\ref{nikulin} these must be isomorphic:
\[
T(X) = U(2)^2 \perp  A_1^2.
\]
This concludes the proof of Theorem \ref{MainTHGenericCase}.

The representation of $\NS(X)$ in \eqref{eq:2D6}
is especially useful for the concept of lattice enhancements
as it allows for writing an abstract isomorphism of discriminant forms as in \eqref{eq:q}.
We will make this isomorphism explicit in \ref{ss:D6}
and exploit it on the level of elliptic fibrations.

\subsection{Lattice Enhancements}\label{ssec:LattEnh}
\subsubsection{General Theory}
\label{ss:latt-enh}

The theory of lattice polarised K3 surfaces as sketched in \ref{ss:intro} predicts
for a given even lattice $L$ of signature $(1,r-1)$
that K3 surfaces admitting a primitive embedding
\[
L\hookrightarrow \NS
\]
come in $(20-r)$-dimensional families
(if $L$ admits a primitive embedding into the K3 lattice $\Lambda = U^3\perp E_8^2$ at all).
Equivalently, on the level of transcendental lattices, 
the primitive embedding has to be reversed for $M$ the orthogonal complement of $L$ in $\Lambda$:
\[
T \hookrightarrow M.
\]
Lattice enhancements provide an easy concept of specifying subfamilies of lattice polarised K3 surfaces.
Namely one picks a vector $v\in M$
of negative square $v^2<0$
and enhances $\NS$ by postulating $v$ to correspond to an algebraic class.
Generically this leads to a codimension one subfamily of K3 surfaces
with transcendental lattice
\[
T= v^\bot \subset M.
\]
The generic N\'eron-Severi lattice arises as primitive closure (or saturation)
\[
\NS = (L + \Z v)' \subset \Lambda.
\]
Explicitly $\NS$ can be computed with the discriminant form.
Namely $v$ induces a unique primitive element in the dual lattice $M^*$.
The resulting equivalence class $\bar v\in M^*/M$
corresponds via
the isomorphism of discriminant forms
$q_M \cong -q_L$ as in  \eqref{eq:q} with an equivalence class $\bar w\in L^*/L$.
Enhancing  $L + \Z v$ by $\bar v + \bar w$ results in a well-defined even saturated lattice 
which exactly gives $\NS$.

Note that presently $M=U(2)^2\perp A_1^2$ has rank equalling its length,
so any primitive vector $v\neq 0$ induces an order 2 element $\bar v$ in $M^*/M$.
In other words, $(L+\Z v)$ has index $2$ in its primitive closure.
If $v$ is assumed to be primitive in $M$, then we find the generic discriminant
of the lattice enhancement 
\begin{eqnarray}
\label{eq:v^2}
\mbox{discr} \NS = (\mbox{discr} \, L) \cdot v^2/4 = -16v^2.
\end{eqnarray}

In the following we want to study the K3 surfaces
corresponding to the divisors $D_\Delta$ in the moduli space
and relate them to Hermann's work \cite{Her}.
To this end, we shall enhance $\NS$ by a primitive representative $\mathbf{y}^*$
as explained in Corollary \ref{OneOrbit} and the following paragraph.

\subsubsection{Example: lattice enhancements by a $-4$ vector}
\label{ss:-4}
 
 We return to our double sextics branched along 6 lines.
Following \ref{ss:latt-enh} we will enhance the  N\'eron-Severi lattice by a $-4$-vector
from $M=U(2)^2\perp A_1^2$.
We consider two ways to do so
which we will soon see to be inequivalent and exhaustive.
Following up on Example \ref{warning}
we shall take either
\[
v_1=(0,0,0,0,1,1) \;\; \text{ or } \;\;v_2=(1,-1,0,0,0,0).
\]
Computing their orthogonal complements in $M$,
we find the generic transcendental lattices of 
the enhanced lattice polarised K3 surfaces:
\begin{eqnarray}
\label{eq:T1}
T_1  = v_1^\bot & = & U(2)^2 \perp  \langle -4\rangle,\\
\label{eq:T2}
T_2  =  v_2^\bot & = & U(2) \perp  A_1^2 \perp  \langle 4\rangle.
 \end{eqnarray}
Note that the first lattice
(which corresponds to $\Delta=1$ by Example \ref{warning}, see also \ref{sss:3.6.9}) is 2-divisible as an even lattice 
while the second lattice (corresponding to $\Delta=4$, see also \ref{ss:4}) certainly is not.
This confirms that these two cases are indeed inequivalent.
In what follows, we will interpret the enhancements in terms of elliptic fibrations.
Along the way, we will verify that any other lattice enhancement by a $-4$-vector
is equivalent to one of the above.

\subsubsection{Interpretation in terms of elliptic fibrations}
\label{ss:D6}

In view of the Picard number, a jacobian elliptic fibration can be enhanced in only 2 ways:
either by a degeneration of singular fibers (changing the configuration of ADE-types)
or by an additional section (which any multisection can be reduced to by \eqref{eq:E}).
For the second alternative, there are usually many possibilities,
distinguished by the height of the section, but also by precise intersection numbers,
for instance encoding the fiber components met.
However, once we fix the discriminant which we are aiming at,
this leaves only a finite number of possibilities.

As an illustration, consider the fibration from \ref{ss:2D6}
exhibiting the representation
\[
\NS(X) = U \perp D_6^2 \perp A_1^2.
\]
Enhancing $\NS$ as in \ref{ss:-4}, we reach a subfamily of lattice polarised K3 surfaces
of Picard number $\rho=17$ and   discriminant $64$ by \eqref{eq:v^2}.
As the discriminant stays the same as before up to sign,
there are only 3 possibilities of enhancement to start with:
\begin{itemize}
\item
2 fibers of type $I_2$ degenerate to $I_4$,
\item
$I_1$ and $I_2^*$ degenerate to $I_3^*$,
\item
or a section $P$ of height  $h(P)=1$.
\end{itemize}
Using the theory of Mordell-Weil lattices from \ref{ss:ell}, 
the third case can be broken down into another 3 subcases,
depending on the precise fiber components met.
Recall that the non-identity components
of $I_n^*$ fibers ($n>0$) are  divided into the near component
(only one component away from the identity component)
and the two far components as visible in the corresponding root diagram of Dynkin type $D_{n+4}$:

\begin{figure}[ht!]
\setlength{\unitlength}{,5mm}
\begin{picture}(100,20)(-60,0)

%
\multiput(23,8)(20,0){4}{\circle*{3}}
\put(23,8){\line(1,0){60}}
\put(83,8){\line(2,1){17}}
\put(83,8){\line(2,-1){17}}
\put(100,16){\circle*{3}}
\put(100, 0){\circle*{3}}
\put(5,7){\small near}
\put(105,14){\small far}
\put(105, -2){\small far}

%

\end{picture}
\end{figure}

An easy enumeration of the possible configurations
reveals the following possibilities for a section $P$ of height $h(P)=1$;
all of them have $P$ perpendicular to the zero section.

\begin{center}
\begin{tabular}{c|c|c}
alternative & $I_2^*$'s & $I_2$'s\\
\hline
(1) & far, far & id, id\\
(2) & far, near & id, non-id\\
(3) & near, near & non-id, non-id
\end{tabular}
\end{center}

\subsubsection{Comparison of enhancements}
\label{ss:compare}

We shall now compare our investigation of the above elliptic fibration 
with the concept of lattice enhancements
 by making the isomorphism \eqref{eq:q} explicit.
  We start by  calculating the discriminant form  of $D_6$. The discriminant group is $(\Z/2\Z)^2$ with generators   $a^\pm\in  D_6\otimes\Q$  represented by elements meeting each one of  the two far nodes $r^\pm$ precisely once and none of the remaining roots $r_1,\dots,r_4$ (enumerated from left to right).
  \begin{figure}[ht!]
\setlength{\unitlength}{,5mm}
\begin{picture}(100,20)(-60,0)

\multiput(23,8)(20,0){4}{\circle*{3}}
\put(23,8){\line(1,0){60}}
\put(83,8){\line(2,1){17}}
\put(83,8){\line(2,-1){17}}
\put(100,16){\circle*{3}}
\put(100, 0){\circle*{3}}

 \put(17,15){$-\half$}
 \put(37,15){$-1$}
 \put(57,15){$-\frac 32$}
  \put(77,15){$-2$}

\put(102,14){$-\frac 32$}
\put(102, -2){$-1$}

\end{picture}
\end{figure}

 The correct rational linear combination $a^+=-(\half r+\frac 32 r^++r^-)$,   $r=r_1+2r_2+3r_3+4r_4$  a root,   is shown in the figure. Of course the combination for $a^-$ is similar and so we find
 $(a^\pm)^2=-\frac 32$ and $a^+\cdot a^-=-1$ so that
 \[ q_{D_6}=\begin{pmatrix}-\frac 32 & -1\\ -1 &-\frac 32\end{pmatrix}.
 \]
The discriminant group of $U(2)$ is also $(\Z/2\Z)^2$ with 
basis 
$e, f$ induced from 
the standard basis of $U(2)$. It follows that 
\[
q_{U(2)}=\begin{pmatrix} 0 & \frac 12 \\ \half & 0 \end{pmatrix}.
\] 
Thus one easily verifies the isomorphism in the standard basis:
\begin{eqnarray*}
q_{U(2)\perp A_1} & \stackrel{\cong}{\longrightarrow} & -q_{D_6\perp A_1} \\
(e,f,g) & \mapsto & (g+e,g+f,e+f+g)
\end{eqnarray*}
Duplicated this directly extends to the isomorphism \eqref{eq:q}.
We continue by computing the impact of the enhancing vectors $v_i$ from \ref{ss:-4}.

Starting out with $v_1$, this vector induces the element $(0,0,1)$ in either copy of 
$q_{U(2)\perp A_1}\hookrightarrow q_{T(X)}$.
In each $q_{D_6\perp A_1}$ this corresponds to the class $(1,1,1)$. 
Thus we obtain an algebraic class meeting each reducible fiber.
A priori this would be a multisection,
but using the group structure it induces a section, necessarily of height $1$, of the third alternative in \ref{ss:D6}.

Next we turn to $v_2$. 
We have
\[
\bar v_2 = ((1,1,0),(0,0,0)) \in (q_{U(2)\perp A_1})^2 \cong q_{T(X)}.
\]
Hence $v_2$ induces the same class in $(q_{D_6\perp A_1})^2$.
This corresponds to an algebraic class meeting only one $I_2^*$ fiber non-trivially.
By inspection of the alternatives in \ref{ss:D6},
this class cannot be a section of height $1$ (which always meets both $I_2^*$ fibers non-trivially),
but it fits in with the degeneration of $I_1$ and $I_2^*$ to $I_3^*$.

\subsubsection{Connection with other enhancements}
\label{ss:connect}

Before returning to the arrangement of the 6 lines,
we comment on the other three possible enhancements of the elliptic fibration in \ref{ss:D6}.
In fact, the freedom of choosing some lines out of the 6 carries over to these elliptic fibrations
endowing $X$ with several different ones of the same shape.
We leave it to the reader to follow the degeneration of singular fibers on the given fibration
through the other elliptic fibrations.
Without too much effort, this enables us to identify all remaining degenerations with the one 
which was shown in \ref{ss:compare} to correspond to
the lattice enhancement by $v_2$.

\subsection{Special Arrangements of the $6$ Lines; Proof of Theorem~\ref{MainTHDegenerations}}\label{ssec:SpecLines}

We are now in the position to investigate the subfamilies of our double sextics
corresponding to the first few divisors on the moduli space of Abelian fourfolds
of Weil type.
In each case, we start from the special arrangement of lines
to fill out the geometric and lattice theoretic details.

\subsubsection{$\Delta=2$}
\label{sss:3.6.1}
   
Recall from \S~\ref{subsec:6Lines} that $D_2$ corresponds to $X_3$.
We now consider the component  $X^{345}_3$, that is, when the lines $\ell_3, \ell_4, \ell_5$ meet in a single point.
On the double covering K3 surface, this results in a triple point
whose resolution requires an additional blow-up.
On the degenerate K3 surface, the original exceptional divisors can still be regarded as perpendicular
(with notation adjusted, see the figure below);
with the lines, however, they do not connect  to a hexagon anymore,
but  to a star  through the additional exceptional component $D$ (Kodaira type $IV^*$):

\begin{figure}[ht!]
\setlength{\unitlength}{.35in}
\begin{picture}(12,5)(-1,-0.4)
\thicklines

\multiput(0,2)(4,0){2}{\circle*{.1}}
\multiput(1,4)(2,0){2}{\circle*{.1}}
\multiput(1,0)(2,0){2}{\circle*{.1}}

\multiput(1,0)(0,4){2}{\line(1,0){2}}
\multiput(0,2)(3,2){2}{\line(1,-2){1}}
\multiput(0,2)(3,-2){2}{\line(1,2){1}}

\multiput(8,2)(1,0){3}{\circle*{.1}}
\multiput(11,1)(0,2){2}{\circle*{.1}}
\multiput(12,0)(0,4){2}{\circle*{.1}}

\put(10,2){\circle*{.2}}

\put(10,2){\line(1,1){2}}
\put(10,2){\line(1,-1){2}}
\put(10,2){\line(-1,0){2}}

\put(-.6,2){$\ell_3$}
\put(7.4,2){$\ell_3$}

\put(3.2,-.3){$\ell_5$}
\put(12.2,-.3){$\ell_5$}

\put(3.2,4){$\ell_4$}
\put(12.2,4){$\ell_4$}

\put(.4,-.4){$E_{35}$}
\put(.15,4){$E_{34}$}
\put(4.2,2){$E_{45}$}

\put(8.8,2.2){$E_3$}
\put(11.2,2.8){$E_4$}
\put(11.2,1){$E_5$}

\put(10.3,1.9){$D$}

\put(5.8,1.9){$\leadsto$}

%
%
%
%
%
%
%
%
%
%
%
%
%
%
%
%
%
%
%
%
%
%
%
%
%
%
%
%
%

%
\end{picture}
\end{figure}

On the standard elliptic fibration from \ref{ss:2D4},
this degeneration causes three $I_2$ fibers to merge to a single additional fiber of type $I_0^*$
(with a 'new' rational curve $E$ as 4th simple component;
compare Figure \ref{fig} where also some rational curves such as $\ell_5, \ell_6$ have been omitted):

\begin{figure}[ht!]
\setlength{\unitlength}{.45in}
\begin{picture}(8,4.2)(-1.5,0)
\thicklines

\multiput(0,1)(2,0){5}{\circle*{.1}}

\multiput(0,3)(2,0){5}{\circle*{.1}}

\multiput(1,2)(6,0){2}{\circle*{.1}}

\multiput(4.5,0)(1,0){1}{\circle*{.1}}
\multiput(4.5,4)(1,0){1}{\circle*{.1}}

\put(0,1){\line(1,1){2}}
\put(6,1){\line(1,1){2}}

\put(0,3){\line(1,-1){2}}
\put(6,3){\line(1,-1){2}}

\put(2,1){\line(1,0){4}}
\put(2,3){\line(1,0){4}}

\put(4,1){\line(0,1){2}}
\put(3.5,2){\line(1,0){1}}

\put(4.5,0){\line(-1,2){.5}}

\put(4.5,4){\line(-1,-2){.5}}

\put(4,1.5){\circle*{.1}}
\put(4,2.5){\circle*{.1}}
\put(3.5,2){\circle*{.1}}
\put(4.5,2){\circle*{.1}}
\put(4,2){\circle*{.1}}

\put(4.1,1.4){$E_4$}
\put(4.1,2.4){$E_3$}

\put(4.6,1.9){$E_5$}

\put(3.1,1.9){$E$}
\put(3.65,1.68){$D$}

\put(4.65,0){$E_{46}$}

\put(4.65,4){$E_{36}$}

\put(.155,.9){$E_{15}$}
\put(1.3,.9){$E_{14}$}

\put(.155,2.9){$E_{16}$}
\put(1.25,2.9){$E_{13}$}

\put(6.155,.9){$E_{24}$}
\put(7.3,.9){$E_{25}$}

\put(6.155,2.9){$E_{23}$}
\put(7.25,2.9){$E_{26}$}

\put(1.2,1.9){$\ell_1$}


\put(7.2,1.9){$\ell_2$}

\put(3.65,.65){$\ell_4$}
\put(4,1){\circle{.2}}

\put(3.65,3.25){$\ell_3$}
\put(4,3){\circle{.2}}

\thinlines
\put(-.25,0.75){\framebox(2.5,2.5){}}
\put(5.75,0.75){\framebox(2.5,2.5){}}

\put(-.25,.4){$u=0$}
\put(7.3,.4){$u=\infty$}

\put(3,1.25){\framebox(2.1,1.5){}}

\end{picture}
\end{figure}

Thus $\NS$ has the index 4 sublattice $U \perp D_4^3 \perp  A_1^3$ -- which is again 2-elementary.
The remaining generators of $\NS(X)$ can be given by the 2-torsion sections.
To decide on the discriminant form, we once more switch  
to the alternative elliptic fibration from \ref{ss:2D6}.
The fibration degenerates as follows: 
in the notation from \ref{ss:2D6}
we have to replace $E_{35}, E_{45}$ by $E_3, E_4$ as components of the $I_2^*$ fibers,
and $E_{34}'$ by $E_5$ as component of one $I_2$ fiber.
Then $E_{56}$ still sits on a second $I_2$ fiber
while $E$ gives yet another one.
Here $D$ induces a 2-torsion section:
visibly it meets both $I_2^*$ fibers at far components.
As for the $I_2$ fibers, it meets the one containing $E_{56}$ at the other  component 
(i.e.~non-identity) and the one containing $E_5$ in this very component (non-identity again).
Since the height of a section is non-negative, 
this already implies that the section $D$ has height $0$;
then the fiber types predict that $D$ can only be 2-torsion.

For completeness we study the fiber containing $E$ in detail.
Since $\ell_6$ is also a section for the standard fibration,
it meets some simple component of the degenerate $I_0^*$ fiber.
Obviously $\ell_6$ does not meet any of $E_3, E_4, E_5$.
Hence $\ell_6$ has to meet $E$.
In conclusion $E$ is the identity component of  the degenerate $I_2$ fiber
of the alternative fibration.
As $D$ meets this fiber trivially, 
we find the orthogonal decomposition
\[
\NS(X) = U \perp \langle D_6^2, A_1^2, D\rangle \perp A_1.
\]
As before, we deduce from the orthogonal summand $A_1$
that the discriminant form takes non-integer values.
Hence by Proposition \ref{nikulin} we deduce that
\begin{equation}
T(X) =  U(2)^2\perp A_1 \;\;\; \text{corresponding to} \;\;\; \Delta=2.
\end{equation}
In the language of lattice enhancements, the subfamily thus arises
from a generator of either $A_1$ summand in the generic transcendental lattice.
Recall that geometrically, this vector corresponds to the extra rational curve $D$ involved in the resolution
of the triple point where three lines come together.
Conversely, we can derive from Proposition \ref{nikulin} again that
the N\'eron-Severi lattice admits several representations purely in terms of $U$ and root lattices
such as 
\[
\NS(X) = U \perp D_4^2\perp E_7.
\]
This concludes the proof of Theorem \ref{MainTHDegenerations} 1. 
\qed \endproof

\subsubsection{$\Delta=4$}
\label{ss:4}

Recall from \S~\ref{subsec:6Lines} that $X_4$ comes from the divisor $\mathbf{X}_{\rm coll}$. Let $\ell$
be the line which contains the collinear points.
Then $\ell$ splits on $X$ as $\pi^* \ell = \ell'+\ell''$.
Let $D=\ell'-\ell''$.
Since $D$ is anti-invariant for the covering involution, 
it defines an algebraic divisor on $X$ which is orthogonal to the
classes specialising from the generic member.
By construction, $\ell'.\ell''=0$ so that $D^2=-4$.
In particular, $D$ is primitive in $\NS(X)$,
and $X$ arises from a lattice enhancement by the $-4$-vector $D$ as in \ref{ss:-4}.
Presently we can even give a $\Z$-basis of $\NS(X)$
by complementing the generic basis by $\ell'$, say.
To compute $\NS(X)$ and $T(X)$ without writing out intersection matrices etc,
we make use of elliptic fibrations again.

For the standard fibration,
it is convenient to choose the collinear points as $P_{12}, P_{34}, P_{56}$.
From the obvious $-2$-curves, each $\ell'$ and $\ell''$ then meets 
exactly the corresponding exceptional divisors $E_{12}, E_{34}, E_{56}$
on $X$.
On the standard fibration from \ref{ss:2D4},
the two singular fibers of type $I_2$ at $E_{34}$ and $E_{56}$ are thus connected by $\ell', \ell''$,
merging to a fiber of type $I_4$.
Note that this indeed preserves the discriminant up to sign while raising the rank by one.

\begin{figure}[ht!]
\setlength{\unitlength}{.45in}
\begin{picture}(8,4.2)(-1.5,0)
\thicklines

\multiput(0,1)(2,0){2}{\circle*{.1}}

\multiput(0,3)(2,0){2}{\circle*{.1}}

\multiput(6,1)(2,0){2}{\circle*{.1}}

\multiput(6,3)(2,0){2}{\circle*{.1}}

\multiput(1,2)(6,0){2}{\circle*{.1}}

\multiput(3,0)(1,0){2}{\circle*{.1}}
\multiput(3,4)(1,0){2}{\circle*{.1}}

\multiput(3.5,1)(2,0){1}{\circle*{.1}}

\multiput(3.5,3)(2,0){1}{\circle*{.1}}

\put(0,1){\line(1,1){2}}
\put(6,1){\line(1,1){2}}

\put(0,3){\line(1,-1){2}}
\put(6,3){\line(1,-1){2}}

\put(2,1){\line(1,0){4}}
\put(2,3){\line(1,0){4}}

\put(3.5,1){\line(0,1){2}}
\put(3.5,2){\line(1,1){.5}}
\put(3.5,2){\line(1,-1){.5}}

\put(4.5,2){\line(-1,1){.5}}
\put(4.5,2){\line(-1,-1){.5}}

\put(3,0){\line(1,2){.5}}
\put(4,0){\line(-1,2){.5}}

\put(3.,4){\line(1,-2){.5}}
\put(4,4){\line(-1,-2){.5}}

\put(4,1.5){\circle*{.1}}
\put(4,2.5){\circle*{.1}}
\put(3.5,2){\circle*{.1}}
\put(4.5,2){\circle*{.1}}

\put(4.1,1.35){$\ell''$}
\put(4.1,2.45){$\ell'$}

\put(4.6,1.9){$E_{56}$}

\put(2.85,1.9){$E_{34}$}

\put(2.3,0){$E_{45}$}
\put(4.15,0){$E_{46}$}

\put(2.3,4){$E_{35}$}
\put(4.15,4){$E_{36}$}

\put(.155,.9){$E_{15}$}
\put(1.3,.9){$E_{14}$}

\put(.155,2.9){$E_{16}$}
\put(1.25,2.9){$E_{13}$}

\put(6.155,.9){$E_{24}$}
\put(7.3,.9){$E_{25}$}

\put(6.155,2.9){$E_{23}$}
\put(7.25,2.9){$E_{26}$}

\put(1.2,1.9){$\ell_1$}


\put(7.2,1.9){$\ell_2$}

\put(3.75,.68){$\ell_4$}
\put(3.5,1){\circle{.2}}

\put(3.75,3.18){$\ell_3$}
\put(3.5,3){\circle{.2}}

\thinlines
\put(-.25,0.75){\framebox(2.5,2.5){}}
\put(5.75,0.75){\framebox(2.5,2.5){}}

\put(-.25,.4){$u=0$}
\put(7.3,.4){$u=\infty$}

\put(2.8,1.25){\framebox(2.4,1.5){}}

\end{picture}
\end{figure}

Switching to the alternative fibration from \ref{ss:2D6},
the classes $\ell', E_{56}, \ell''$ which correspond to the root lattice $A_3$ 
remain orthogonal to the fibers $D_1, D_2$.
Generically they are therefore contained in a fiber of type $I_4$,
merging the $I_2$ fibers generically at $E_{56}$ and at $E_{34}'$.
That is, $\NS(X)=U\perp D_6^2\perp A_3$.
By \ref{ss:compare}, \ref{ss:connect} 
we can thus verify
that $X$ arises from the lattice enhancement by $v_2$ with transcendental lattice 
$T(X)=U(2)\perp \langle 4\rangle\perp A_1^2$ as stated in Theorem \ref{MainTHDegenerations} 2.

\subsubsection{$\Delta=1$}

As the key part of this subsection,
we now come to the case $\Delta=1$
which will cover almost the rest of this section up to \ref{sss:3.6.9}. 
Recall from \S\ref{subsec:6Lines} that we have the
divisor $XQ\subset X$ of 6-uples of lines tangent to a fixed conic. This divisor can be identified with $D_1$ as follows
from \cite[Prop. 2.13.4]{Mats}. Indeed, that  hyperplane $XQ$ is exactly the hyperplane orthogonal to our $v_1$ (see \S~\ref{ss:-4}).  

This can also be read off directly from the fact that
$X$ is a Kummer surface. Indeed, by \cite[\S~0.19]{MSY} the surface $X$   arises from the jacobian of the genus 2 curve
which is the double cover of the conic branched along the six intersection points with the lines.
This gives $T(X)=U(2)^2+\langle-4\rangle$ in agreement with the lattice enhancement by $v_1$. We shall confirm this from our methods using elliptic fibrations.

The above argument, however, gives no information about the extra algebraic class
needed to generate $\NS(X)$ over $\Z$.
To overcome this lack of a generator,
we shall work geometrically with elliptic fibrations,
starting with the alternative fibration.
Going backwards in our constructions,
we first develop the corresponding section on the standard fibration
and then interpret this in terms of  a certain conic in $\PP^2$
which splits on $X$.
Finally we confirm our geometric arguments by providing explicit equations.
Throughout we do not use any information about the Kummer surface structure.

 First, denoting the conic  by $C$, we have a splitting $C=C_1+C_2$ on $X$. 
As in \ref{ss:4}, the divisor $D=C_1-C_2$ is anti-invariant for the covering involution
and therefore orthogonal to the rank 16 sublattice of $\NS(X)$ generated by
the classes of the lines and the exceptional divisors.
The subtle difference, though, is that $D$ is in fact $2$-divisible in $\NS(X)$
since $C\sim 2H$:
\[
\frac 12 D = H-C_2 \in\NS(X).
\]
Since $D^2=-16$, we find that the latter class has square $-4$,
hence we are indeed  confronted with a lattice enhancement as in \ref{ss:-4}.

\subsubsection{From alternative to standard fibration}
\label{sss:alt-st}

In \ref{ss:D6} we interpreted lattices enhancements in terms of the alternative elliptic fibration
from \ref{ss:2D6}.
By \ref{ss:compare}, \ref{ss:connect}
it is
alternative (3)  which corresponds to the lattice enhancement by $v_1$.
In detail, the alternative fibration admits a section $P$ intersecting the following fiber components (see Figure \ref{Fig:2} for the resulting diagram of $-2$-curves):

\begin{center}
\begin{small}
\begin{tabular}{c|cc|cc}
singular fiber & $D_1=I_2^*$ & $D_2=I_2^*$ &$I_2$ & $I_2$  \\
\hline
component met by $\ell'$ & $E_{15}$ & $E_{25}$ & opposite $E_{34}'$ & opposite $E_{56}$\\ 
\end{tabular}
\end{small}
\end{center}

On the standard fibration, $P$ defines a multisection
whose degree is not immediate.
Here we develop a backwards engineering argument
to prove that the degree is actually $1$,
i.e.~$P$ is a section for both fibrations.

A priori the degree $d$ of  the multisection $P$ need not be $1$ on the standard fibration,
but $P$ always induces a section $P'$ of height $1$.
The essential point of our argument is that we can read off from the alternative fibration 
which fiber components are not met by the multisection on the standard fibration.
For each singular fiber this leaves only one fiber component
with intersection multiplicity depending on the degree $d$.
But then we can use the group structure to determine
which fiber component will be met by the induced section $P'$.
Thanks to the specific singular fibers,
the argument only depends on the parity of  $d$:

\begin{center}
\begin{small}
\begin{tabular}{c|cc|cccccc}
singular fiber & $I_0^*$ & $I_0^*$ &$I_2$ & $I_2$ & $I_2$ & $I_2$ & $I_2$ & $I_2$ \\
\hline
\hline
comp's met by $P$ & $\begin{matrix} E_{15},\;\;\;\;\\ (d-1)  E_{14}\end{matrix}$ & 
$\begin{matrix}E_{25},\;\;\;\;\\ (d-1)  E_{23}\end{matrix}$ & $d  E_{34}$ & $d  E_{35}'$ & $d  E_{36}$ & $d  E_{45}'$ & $d  E_{46}$ & $d  E_{56}'$\\ 
\hline
comp met by $P'$ & $E_{16}$ & $E_{25}$ & $E_{34}$ & $E_{35}$ &
$E_{36}$ & $E_{45}'$ & $E_{46}'$ & $E_{56}'$\\
for even $d$ & non-id & non-id & id & id & id & id & id & id\\
\hline
comp met by $P'$ & $E_{15}$ & $E_{25}$ & 
$  E_{34}$ & $  E_{35}'$ & $  E_{36}$ & $  E_{45}'$ & $  E_{46}$ & $  E_{56}'$\\
for odd $d$ & non-id & non-id & id & non-id & id & id & non-id & id
\end{tabular}
\end{small}
\end{center}

Note that for even $d$
the section $P'$ would have even height by inspection of the fiber components met,
contradicting $h(P')=h(P)=1$.
Hence $d$ is odd, and the induced section $P'$ meets the fiber components 
indicated in the last two rows of the table.
With this section at hand, we can complete the circle:
namely
$P'$ defines a section for both fibrations,
of exactly the same shape as $P$, hence $P'=P$ (and $d=1$).

For later reference,
we point out the symmetry in the fiber components met by $P$ on the standard fibration:
on the $I_0^*$ fibers, it is exactly those met by the section $\ell_5$,
while on the $I_2$ fibers it is exactly those not met by $\ell_5$.
This symmetry is essential for the section to be well-defined 
as it ensures that adding a 2-torsion section to $P$ will always result in a section of height $1$
(compare \ref{sss:eqn}).

\subsubsection{From standard fibration to double sextic}
\label{sss:st-6}

On the double sextic model, 
the section $P$ arises from a curve $Q$ in $\PP^2$
which splits into 2 rational curves $Q_1, Q_2$ on $X$.
Here we give an abstract description of $Q$ and its components on $X$
based on the geometry of the elliptic fibrations.

From the alternative fibration
we know that $P$ is perpendicular to the lines $\ell_i$ for $i\neq 5$.
On the other hand, the standard fibration reveals by inspection of the above table
which exceptional curves intersect $P$:
\[
\text{exactly } \;\; E_{15}, E_{25}, E_{34}, E_{36}, E_{46}\;\; \text{ plus possibly } \;\; E_{12}.
\]
The latter is the only exceptional curve which is not visible
as section or fiber component on the standard fibration.
The remaining two intersection numbers can be computed as follows:
regarding $\ell_5$ as a 2-torsion section of the standard fibration,
the height pairing $\langle P,\ell_5\rangle=0$ implies by virtue of the fiber components met
that $P.\ell_5=0$.
As for $E_{12}$,
consider the auxiliary standard fibration defined by $u'=\ell_1/\ell_5$.
Then $P$ defines a section for this fibration as well, 
as it meets the fiber 
\[
(u')^{-1}(\infty) = 2\ell_5+E_{25}+E_{35}+E_{45}+E_{56}
\]
exactly in $E_{25}$ (transversally) by the above considerations.
Looking at the fiber
\[
(u')^{-1}(0) = 2\ell_1 + E_{12}+E_{13}+E_{14}+E_{16}
\]
we deduce $P.E_{12}=1$ from the fact that $P$ does not intersect the other fiber components.

Turning to the double sextic model of $X$,
$P$ defines a curve meeting 
$E_{12}, E_{15}, E_{25}, E_{34}, E_{36}, E_{46}$ transversally,
but no other exceptional curves $E_{ij}$ nor any of the $\ell_i$.
On the base $\PP^2$, this curve necessarily corresponds to a conic $Q$ 
through the six underlying nodes.
On $X$, this conic splits into two disjoint rational curves $Q_1, Q_2$
where $Q_1=P$, say, 
and $Q_2$ corresponds to the section $-P$ on the elliptic fibrations.

\subsubsection{Explicit equations}
\label{sss:eqn}

We start out with the general equation of a jacobian elliptic K3 surface  $X$ with 
singular fibers of type $I_0^*$ twice and 6 times $I_2$ over some field $k$ of characteristic $\neq 2$.
Necessarily this comes with full $2$-torsion.
Locating the $I_0^*$ fibers at $t=0,\infty$,
we can write
\begin{eqnarray}
\label{eq:leg}
X:\;\;\; tw^2 = x (x-p(t)) (x-q(t))
\end{eqnarray}
where $p, q\in k[t]$ have degree 2.
Here the $I_2$ fibers are located at $p=0, q=0$ $p=q$.
Up to symmetries, there are only two ways to 
endow the above fibration with a section of height $1$.
Abstractly, restrictions are imposed by the compatibility with the 2-torsion sections.
On the one hand, there is the symmetric arrangement encountered in \ref{sss:alt-st}.
This will be investigated below.
On the other hand, an asymmetric arrangement can be encoded, for instance,
in terms of the standard fibration
by a section with the following intersection behaviour:

\begin{center}
\begin{small}
\begin{tabular}{c|cc|cccccc}
singular fiber & $I_0^*$ & $I_0^*$ &$I_2$ & $I_2$ & $I_2$ & $I_2$ & $I_2$ & $I_2$ \\
\hline
component met & $E_{15}$ & $E_{26}$ & 
$  E_{34}$ & $  E_{35}$ & $  E_{36}$ & $  E_{45}$ & $  E_{46}$ & $  E_{56}'$\\
\end{tabular}
\end{small}
\end{center}

To see that the arrangements do indeed generically describe different K3 surfaces,
we switch to the alternative fibration from \ref{ss:2D6} for one final time.
Here the section $P$ with intersection pattern as in the above table
induces a bisection meeting far and near component of $D_1$
and far and identity component of $D_2$.
Using the group structure the induced section intersects both $I_2^*$ fibers in a far component.
In terms of \ref{ss:D6} this corresponds to alternative (1)
which was shown in  \ref{ss:compare}, \ref{ss:connect}
to differ from alternative (3) which underlies the section arrangement in \ref{sss:alt-st}.

\smallskip

We shall now continue by deriving equations admitting a section of height $1$
as encountered in \ref{sss:alt-st}.
In agreement with this,
we model the section $P$ to intersect the same components of the $I_0^*$ fibers
as the $2$-torsion section $(0,0)$.
In terms of the RHS of \eqref{eq:leg}, these correpond to the factor $x$.
The section $P$ therefore takes the shape 
\[
P=(at,\hdots)\;\; \text{ for some constant } \;\; a\in k.
\]
Upon subsituting into \eqref{eq:leg},
we now require the other two factors on the RHS 
to produce the same quadratic polynomial up to a constant.
Concretely this polynomial can be given by
\begin{eqnarray}
\label{eq:g=}
g = at-p.
\end{eqnarray}
Then we consider the codimension 1 subfamily of all $X$
such that there exists $b\in k, b\neq 1$ such that
\begin{eqnarray}
\label{eq:q=}
q = at-bg.
\end{eqnarray}
By construction, these elliptic K3 surfaces admit the section 
\[
P=(at, \sqrt{ab}g).
\]
This section has height $1$:
not only does it meet both $I_0^*$ fibers non-trivially
thanks to our set-up,
but also the $I_2$ fibers at $p-q=(b-1)g=0$
while being perpendicular to the zero section.
As a whole,
the family of K3 surfaces can be given by letting $g,a,b$ vary
and $p,q$ depend on them as above.
There are still normalisations in $t$ and in $(x,w)$ left
which bring us down to the 3 moduli dimensions indeed.

\subsubsection{Standard fibration reflecting the 6 lines}
\label{sss:eqn-6}

We shall now translate the above considerations
to the standard fibration
as it comes from the 6 lines in \ref{ss:2D4}.
With $u=\ell_1/\ell_2$, we naturally have an equation
\begin{eqnarray}
\label{eq:3-6}
uw^2 = \ell_3\ell_4\ell_5\ell_6.
\end{eqnarray}
Here we can use $u$ to eliminate $x$, say,
so that after clearing denominators the expressions $\ell_i$ are separately linear in both $y$ and $u$.
In particular, the family of elliptic curves over $\PP^1_u$ becomes evident,
with 2-torsion sections given by $\ell_i=0\; (i=3,4,5,6)$.

For ease of explicit computations, we normalise the lines to be
\[
\ell_1=x, \ell_2=y, \ell_3=x+y+z, \ell_4=a_1x+a_2y+a_3z,  \ell_5=z, \ell_6=b_1x+b_2y+b_3z.
\]
Working affinely in the chart $z=1$,
equation \eqref{eq:3-6} readily takes the shape of a twisted Weierstrass form
\[
uw^2 = ((u+1)y+1)((a_1u+a_2)y+a_3)((b_1u+b_2)y+b_3).
\]
Standard variable transformations take this to:
\[
uw^2 = (y+(a_1u+a_2)(b_1u+b_2)) (y+a_3(u+1)(b_1u+b_2))(y+b_3(u+1)(a_1u+a_2)).
\]
Translating to  the shape of \eqref{eq:leg} and solving for \eqref{eq:g=}, \eqref{eq:q=}, we find 
\[
b_1=ba_1b_3/(a_3+(b-1)a_1), \;\;\; b_2 = ba_2b_3/(a_3+(b-1)a_2)
\]
with $a=-ba_3b_3(a_1-a_2)^2/[(a_3+(b-1)a_1)(a_3+(b-1)a_2)]$.

\subsubsection{Conics on the double sextic}

Finally we can trace back the section $P$ to the double sextic model.
Step by step, it leads to the following conic in the affine chart $z=1$:
\begin{eqnarray*}
Q & = &
-a_1^2 x^2 a_2+a_1^2 x^2 a_3+a_1^2 x^2 b a_2-a_1^2 x a_2 y+x a_1^2 a_2 b+a_1^2 x a_2 y b\\
&&
-a_1 x a_2^2 y
 +2 a_3 x y a_1 a_2-x a_1^2 a_2+x a_3 a_1^2+a_1 x a_2^2 y b+a_1 a_2^2 y b\\
&&-a_1 a_2^2 y+a_3 y a_2^2
-a_2^2 y^2 a_1+a_2^2 y^2 a_3+a_2^2 y^2 b a_1.
\end{eqnarray*}
One directly verifies that $Q$ indeed passes through the nodes
$P_{12}, P_{15}, P_{25},$
$P_{34},$
 $P_{36}, P_{46}$ in $\PP^2$ as in \ref{sss:st-6}.
Hence $Q$ splits into two components on the double sextic $X$
one of which is $P$.

We conclude this paragraph by verifying that the subfamily of double sextics 
constructed in \ref{sss:eqn-6}
does in fact admit a conic which is tangent to each of the 6 lines.
For this purpose,
set
\[
\alpha=a_1(a_2-a_3), \beta = a_2(a_3-a_1), \gamma = a_3(a_1-a_2).
\]
Then the conic in $\PP^2$ given by
\[
\alpha^2x^2+\beta^2y^2+\gamma^2z^2-2(\alpha\beta xy+\alpha\gamma xz+\beta\gamma yz)=0
\]
meets each of the 6 lines $\ell_i$ tangentially.

\subsubsection{Conclusion}
\label{sss:3.6.9}

By comparison of moduli dimensions, it follows conversely
that the K3 surfaces $X_1$ with a conic tangent to each of the 6 lines of the branch locus
also admits a conic through a selection of 6 nodes as above.
From \ref{sss:alt-st} we therefore deduce
that $X_1$ generically arises from $X$ via the lattice enhancement by the vector $v_1$;
that is, by \ref{ss:-4}
\[
T(X_1) = U(2)^2 \perp \langle-4\rangle.
\]
The N\'eron-Severi lattice $\NS(X_1)$ is thus generically generated by the sublattice
$U\perp D_6^2\perp A_1^2$ coming from $X$
enhanced by the section $P$ from \ref{sss:alt-st}.
The simple representation of $\NS(X_1)=U\perp D_4\perp D_8\perp A_3$ in Theorem \ref{MainTHDegenerations}
is derived from the above fibration
by switching to yet another jacobian elliptic fibration as depicted below.

\begin{figure}[ht!]
\setlength{\unitlength}{.45in}
\begin{picture}(10,4)(-1,-1)
\thicklines

\multiput(0,0)(0,1.5){2}{\circle*{.1}}

\multiput(.75,.75)(1,0){3}{\circle*{.1}}

\multiput(3.5,0)(0,1.5){2}{\circle*{.1}}

\multiput(4.25,-.75)(0,3){2}{\circle{.1}}

\multiput(5,0)(0,1.5){2}{\circle*{.1}}

\multiput(5.75,.75)(1,0){3}{\circle*{.1}}

\multiput(8.5,0)(0,1.5){2}{\circle*{.1}}

\put(4,1.25){\circle*{.1}}
\put(4,.25){\circle*{.1}}

\qbezier(4,1.25)(3.8,.75)(4,.25)
\qbezier(4,1.15)(4.2,.75)(4,.25)

\put(4.5,1.25){\circle*{.1}}
\put(4.5,.25){\circle*{.1}}

\qbezier(4.5,1.25)(4.3,.75)(4.5,.25)
\qbezier(4.5,1.25)(4.7,.75)(4.5,.25)

%
%
%
%
\put(0,0){\line(1,1){.75}}
\put(0,1.5){\line(1,-1){.75}}

\put(.75,.75){\line(1,0){2}}

\put(5.75,.75){\line(1,0){2}}

\put(8.5,0){\line(-1,1){.75}}
\put(8.5,1.5){\line(-1,-1){.75}}

\put(2.75,.75){\line(1,1){1.5}}
\put(2.75,.75){\line(1,-1){1.5}}

\put(5.75,.75){\line(-1,1){1.5}}
\put(5.75,.75){\line(-1,-1){1.5}}

\put(4.25,-.75){\line(1,4){.25}}
\put(4.25,-.75){\line(-1,4){.25}}

\put(4.25,2.25){\line(1,-4){.25}}
\put(4.25,2.25){\line(-1,-4){.25}}

\put(5.75,.75){\circle{.2}}

\put(4.35,2.35){$O$}
\put(4.35,-1.05){$P$}
\put(5.85,.85){$O'$}

\thinlines
\put(4.125,-1.2){\framebox(1.2,1.6){}}
\put(5.4,-1.1){$A_3$}

\put(6.5,-.25){\framebox(2.25,2){}}
\put(8.8,-.15){$D_4$}
%

\put(-.25,-.25){\line(1,0){3}}
\put(-.25,-.25){\line(0,1){3}}
\put(-.25,2.75){\line(1,0){5.5}}
\put(5.25,2.75){\line(0,-1){1.375}}
\put(5.25,1.375){\line(-1,0){.875}}
\put(2.75,-.25){\line(1,1){1.625}}

\put(-.75,2.3){$\tilde D_8$}

\end{picture}
\caption{From alternative fibration to $U\perp D_4\perp D_8\perp A_3$ on $X_1$}
\label{Fig:2}
\end{figure}
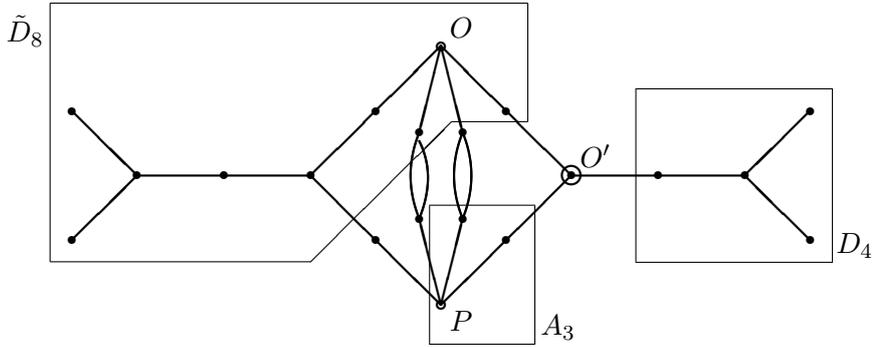

\subsubsection{$\Delta=6$}

Our aim for the final bit of this section
is to understand the geometry of the K3 surfaces for $D_6$, i.e.~the case $\Delta=6$
(the final Heegner divisor singled out in \cite{kondo}).
By Corollary \ref{OneOrbit} and the discussion succeeding it,
this corresponds to a lattice enhancement by a vector $v\in T(2)$ of square $v^2=-6$.
Here we choose the primitive representative
\[
v=(1,-1,1) \text{ in one copy of } U(2)\perp A_1\subset T(2),
\]
augmented by zeroes in $T(2)$.
Then $v/2$ defines a class in the discriminant group $T(2)^*/T(2)$
which via the isomorphism in \ref{ss:compare}
maps to the class 
\[
(0,0,1)\in(D_6\perp A_1)^*/(D_6\perp A_1)\hookrightarrow {\NS}^*/\NS,
\]
augmented by zeroes in ${\NS}^*/\NS$.
The N\'eron-Severi lattice
$\NS$ is thus enhanced by a divisor which only meets one reducible fiber of the alternative fibration
in a non-identity component
(corresponding to $A_1$).
If this divisor were a section,
then it would have height $h\geq 4-1/2=7/2$ by \ref{ss:ell}, but certainly not $3/2$.
Hence the lattice enhancement can only result in a fiber degeneration 
\[
A_1 \rightsquigarrow A_2
\]
on the alternative fibration.
Thus we find the enhanced N\'eron-Severi lattice
\begin{eqnarray}
\label{eq:6}
{\NS}' = U \perp 2D_6\perp A_1\perp A_2
\end{eqnarray}
in agreement with the generic transcendental lattice of the enhanced subfamily,
\[
T(X) = U(2) \perp A_1^2 \perp \langle 6\rangle  \;\;\; \text{corresponding to} \;\;\; \Delta=6.
\]
In order to determine the corresponding special curve in $\PP^2$ explicitly,
we assume without loss of generality
that the $I_2$ fiber with $E_{56}$ as non-identity component degenerates to Kodaira type $I_3$.
That is, there are two other smooth rational curves $D_1, D_2$ as fiber components.
Both give sections of the standard fibration,
meeting exactly the following fiber components:

\begin{center}
\begin{small}
\begin{tabular}{c|cc|cccccc}
singular fiber & $I_0^*$ & $I_0^*$ &$I_2$ & $I_2$ & $I_2$ & $I_2$ & $I_2$ & $I_2$ \\
\hline
component met & $E_{14}$ & $E_{23}$ & 
$  E_{34}$ & $  E_{35}'$ & $  E_{36}$ & $  E_{45}'$ & $  E_{46}$ & $  E_{56}'$\\
& non-id & id & id & non-id & id & id & non-id & non-id
\end{tabular}
\end{small}
\end{center}

 \begin{figure*}[ht!] \includegraphics[width=10cm]{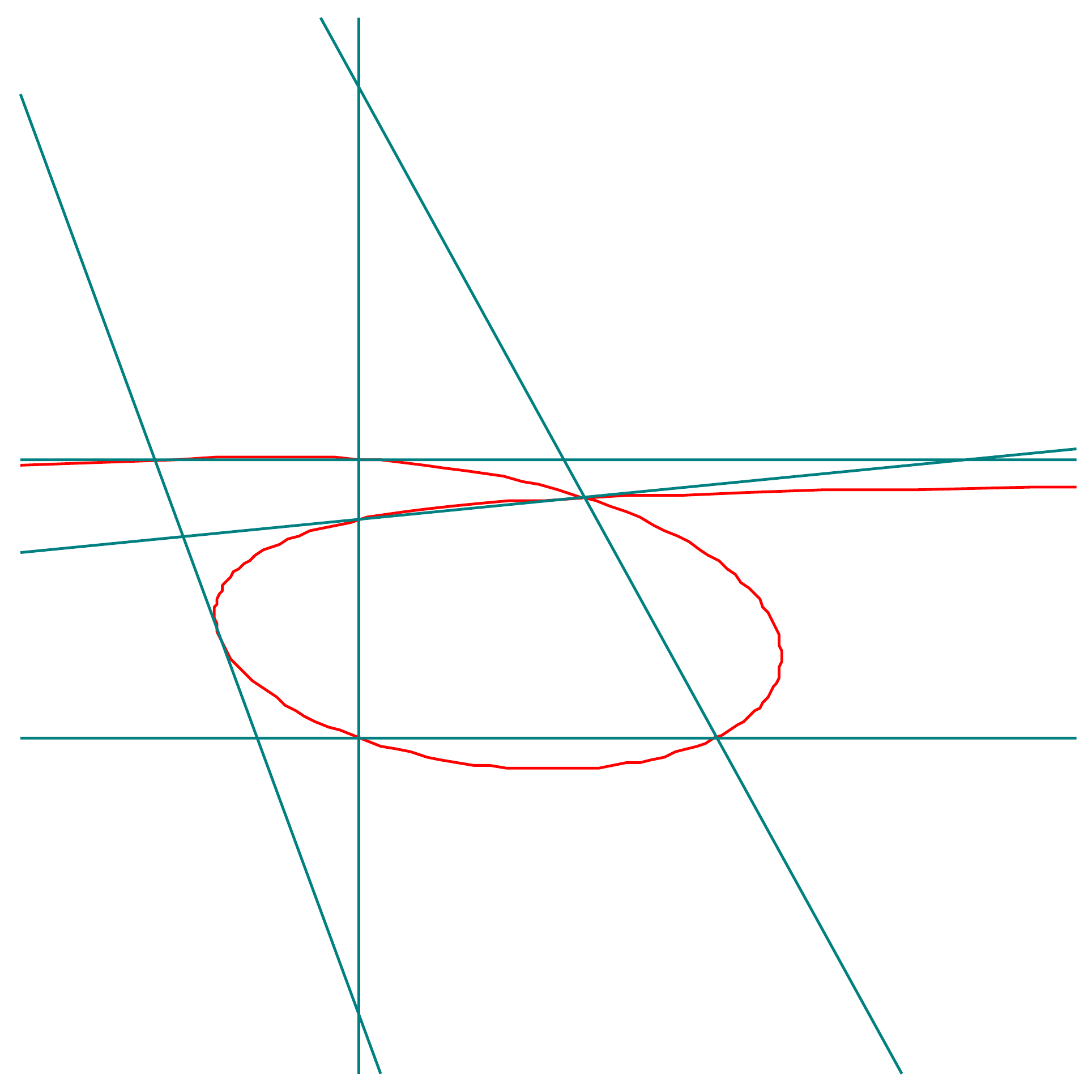}
 \caption{The 6 lines and the singular cubic}
 \end{figure*} 

One directly checks that the height pairing from \ref{ss:ell} gives $h(D_1)=h(D_2)=3/2$; 
in fact the sections are inverse to each other,
since they meet the same fiber components and $\langle D_1,D_2\rangle=-3/2$.
The intersection numbers with all other rational curves from the line arrangement
are zero except for $E_{12}$ and $\ell_5$
which are neither visible on the alternative fibration
nor fiber components of the standard fibration.
Here, since $\ell_5$ defines a 2-torsion section of the standard fibration,
the height pairing 
\[
0 = \langle D_1,\ell_5\rangle= 2 - D_1.\ell_5 - \underbrace{1/2}_{I_0^* \text{ at } E_{14}, E_{15}} - \underbrace{1/2}_{I_2 \text{ at } E_{46}}  \;\; \text{ gives } \;\; D_1.\ell_5=1,
\]
and likewise for $D_2$.
As for $E_{12}$, arguing with an auxiliary standard fibration such as the one induced by $\ell_1/\ell_4$,
we find 
\[
D_1.E_{12}=D_2.E_{12}=2.
\]
It follows that $D_1$ and $D_2$ correspond to a cubic curve $C\subset \PP^2$ of the following shape:
\begin{itemize}
\item
with a singularity at the node underlying $E_{12}$,
\item
through the nodes underlying $E_{14}, E_{23}, E_{34}, E_{36}, E_{46}, E_{56}$,
 \item
 meeting $\ell_5$ tangentially in a smooth point.
 \end{itemize}
 
 \section{The Kuga-Satake Construction}\label{sec:KS}

\subsection{Clifford Algebras} Let $V$ be a finite dimensional $k$-vector space equipped with  a non-degenerate bilinear form  $q$. Its tensor algebra is $\mathsf{T} V=\oplus _{p\ge 0}V^{\otimes p}$, where the convention is that $V^{(0)}=k$.
Recall   that the   \textbf{Clifford algebra} is the following quotient algebra of this algebra:
\[ 
\cl {}{V} =\cl {} {V,q}  := \mathsf{T} V /  \text{\rm ideal generated by }  \sett{x\otimes x- q(x,x)\cdot 1 }{x\in V} . 
\]  
Then we have $xy+yx= 2 q(x,y)$ in the algebra $\cl{} V$; in particular $x$ and $y$ anti-commute whenever they are orthogonal.

The Clifford algebra has  dimension $2^n$ where $n=\dim _k V$. Let us make this explicit for $k=\Q$. Then $Q$ can be    diagonalised   in  some  basis, say  $\set{e_1,\dots,e_n} $. Consider $\mathbf{a}=(a_1,\dots,a_n)\in \bF_2^n$. Taking all $2^n$ possibilities, we find a basis for $\cl {}V $ :
\[
e^{\mathbf a}:= e_1^{a_1}\cdots e_n^{a_n}.
\]
The even Clifford algebra $\cl  + V$ is generated by those $e^{\mathbf a}$ for which $\sum a_j$ is even. 

To describe Clifford algebras certain quaternion algebras play a role.  Let $F$ be a field and $a,b\in F^\times$. The quaternion algebra $(a,b)_F$ over a field $F$ has an $F$--basis
$\set{1, \mathbf{i}, \mathbf{j} ,\mathbf{k}}$ such that $\mathbf{i}^2=a, \mathbf{j}^2=b,\mathbf{i} \mathbf{j}=-\mathbf{j}\mathbf{i}=\mathbf{k}$. 
The Clifford algebra   $\cl {} {\qf{a}\perp\qf{b}}$, $a,b\in \Q^\times$  is isomorphic to $(a,b)_\Q$ while $\cl{+} {\qf{a}\perp\qf{b}}=\Q(\sqrt{-ab})$.  One can see  (cf. \cite{Schar}) that for rank $3$, $4$ and $5$ the results are:
\begin{lemma} \label{CalcGem} Suppose $Q=$ \text{\rm diag}$(a_1,\dots,a_m)$.  Put $d=(-)^m a_1\cdots a_m$. Then
\begin{enumerate}
\item For $m=3$ we have $\cl {+} Q= (-a_1a_2,-a_2a_3)_\Q$;

\item For $m=4$ we have  $\cl {+} Q=  (-a_1a_2,-a_2a_4)_\Q\otimes F$ with $F=\Q\sqrt{d}$;
\item For $m=5$ we have  
$\cl {+} Q=   (-a_1a_2,-a_2a_3)_\Q\otimes_\Q(a_1a_2a_3a_4,-a_4a_5)_\Q$.
\end{enumerate}
\end{lemma}

\subsection{From Certain Weight $2$ Hodge Structures to Abelian Varieties} 
Next suppose that $(V,q)$ carries a  weight $2$ Hodge structure  polarized by $q$ with $h^{2,0}=1$.  
Then    $V^{2,0}\oplus V^{0,2} $ is the complexification of a real plane $W\subset V$ carrying a Hodge substructure and $q$ polarizes it. Then $b(x,y):= -q(x,  y)$ is a  metric on this plane  since  $C=-1$ is the  Weil-operator
\footnote{Recall that $C$ is defined by $C|H^{p,q}=\ii^{p-q}$.} of this Hodge structure. A choice of orientation for $W$ then defines a unique almost complex structure  which is the rotation over $\pi/2$ in the positive direction. Equivalently, this almost complex structure is determined by any positively oriented orthonormal basis $\set{f_1,f_2}$ for $W$. Such a choice also defines an almost complex structure $J=f_1f_2$ on $\cl +V  $   
since $f_1f_2f_1f_2=  -f_1^2f_2^2 =-1$. Then  $J$ defines a  weight $1$ Hodge structure: the eigenspaces of $J$ for the eigenvalues $\pm\ii$ are the Hodge summands $H^{1.0}$, respectively $H^{0,1}$.  

It turns out that the Hodge structure is polarized by a very natural skew form 
\[
E:  \cl + V \times \cl + V\to \Q,\quad (x,y)\mapsto \tr (\epsilon \iota (x)y)
\]
built out of the canonical involution
\[
\iota: \cl + V \to \cl + V,\quad e_1^{a_1}\cdots e_n^{a_n}\mapsto e_1^{a_n}\cdots e_n^{a_1}, 
\]
the trace map
\[
\tr: \cl + V \to \Q,\quad c\mapsto \tr(R_c),\, R_c: x \mapsto cx,\,\text{(left multiplication by } c)
\]
and $\epsilon \in \cl + V$ any element with $\iota\comp \epsilon =-\epsilon$, for instance $\epsilon=e_1e_2$.
If we set 
\[
U:= \text{ vector space dual to } \cl + V,
\]
any  choice of a free $\Z$-module $U_\Z$  makes $U/U_\Z$ into a complex torus which is polarized by $E$. This  Abelian variety by definition is the \textbf{Kuga Satake variety} and is denoted by $A(V,q)$. It is an Abelian variety of dimension $2^{n-2}$, half of the real dimension of $\cl + V$. 

Recall (e.g. \cite{sata}):
\begin{thm}\label{endsdecompo} One has $ \cl{+}{V}\subset \End_\Q(A(V,q))$. If 
\begin{equation*}
\cl{+}{V}=M_{n_1}(D_1)\times\cdots\times M_{n_d}(D_d),\, D_j  \text{ division algebra, } j=1,\dots,d,
\end{equation*}
then we have a decomposition into simple polarized Abelian varieties (here $\sim$ denotes isogeny):
\begin{eqnarray}
\label{eq:5.2}
A(V,q) \sim A_1^{n_1}\times\cdots \times A_d^{n_d},\, A_j \;\;  \text{with } D_j\subset \End_\Q(A_j), j=1,\dots,d.
\end{eqnarray}
\end{thm}

\subsection{Abelian Varieties of Weil Type}
 
 In this note we are mainly interested in the following two examples which arise as certain Kuga-Satake varieties. The first class of examples is related to Abelian varieties in the full moduli space $\tilde{\mathbf{M}}$, and they are described by the following proposition.

\begin{prop}[\protect{\cite[Theorem 6.2]{Lom}}] Let $V$ be a rational vector space of dimension $6$ and let $q=\qf {2}\oplus \qf {2}\oplus \qf {-2}\oplus \qf {-2}\oplus  \qf {-a}\oplus \qf {-b}$,   $a,b\in \Q^+$. We put $\alpha=\ii \sqrt{ab}$.

Suppose that $(V,q)$ is a polarized weight $2$ Hodge structure with $h^{2,0}=1$. Then the Kuga-Satake variety $A(V,q)$ is a $16$-dimensional Abelian variety of $\Q(\alpha)$-Weil-type. For a generic such Hodge structure $\End_\Q(A(V,q))=M_4(\Q(\alpha))$.
\end{prop}
In the situation of  Theorem \ref{endsdecompo} we have $d=1$ and $n_1=4$, i.e. : 
\[
A(V,q)  \sim B^4, \, B \text{ simple 4-dim. Abelian variety with } \Q(\alpha)\subset \End_\Q(B). 
\]
On the other hand, the procedure of  \S~\ref{sec:Tconstruct}  describes how one may associate to any Abelian variety $A$ of dimension $4$ of  $\Q(\ii)$--Weil-type a polarized Hodge substructure $T(A)$ of $H^2(A)$  with $h^{2,0}=1$. We have:
\begin{prop}[\protect{\cite[Theorem 6.5]{Lom}}]  \label{Lombardo} For an Abelian variety $A$ of dimension 4 which is of $\Q(\ii)$--Weil type we have
\[A(T(A))  \sim A^4,\]
in other words, the Kuga-Satake procedure applied to  $T(A)$  gives back the original Abelian variety $A$ up to isogeny.
\end{prop}

The second class relates to the Abelian varieties in the hypersurfaces $D_\Delta$ of the moduli space $\mathbf{M}$.   We recall equation ~\eqref{eqn:OverQ}  where we found   the generic transcendental subspace
 $(T_\Delta)_\Q\subset (T(A_\Delta))_\Q$ of such an Abelian variety $A_\Delta$. 

 We have: 
\begin{thm} \label{KSSpecialCase}
Let $T_\Delta= \qf {2\Delta}\perp   U\perp \qf {-2}^2$ be a polarized Hodge structure of type $(1,3,1)$  and let $A(T_\Delta)$ be its associated Kuga-Satake variety. If  $A_\Delta$ is any   Abelian variety with moduli point in  
the hypersurface  $D_\Delta$ of the moduli space $\mathbf{M}$ such that 
 $T(A_\Delta)=T_\Delta$ as polarized $\Q$--Hodge structures, then  we have an isogeny
\[A(T_\Delta)  \sim A_\Delta^2, \quad  (-1,\Delta)_\Q\subset \End_{\Q}\left(A_\Delta  \right).
\]
\end{thm}
\proof
 We consider the Clifford algebra associated to the Hodge structure $T_\Delta$, since $U\cong_{\Q}\langle 2\rangle\oplus\langle -2\rangle$ the quadratic form is  $Q$=diag$(2,-2,-2,-2,2\Delta)$.
 For the corresponding Clifford algebra
  we deduce from Lemma~\ref{CalcGem} that 
\begin{eqnarray*}
\cl + Q &\cong & (4,-4)_\Q\otimes(-16,4\Delta)_\Q\\
&\cong & (1,-1)_\Q\otimes (-1,\Delta)_\Q\cong M_2((-1,\Delta)_\Q). 
\end{eqnarray*}
From Thm.~\ref{endsdecompo} the Kuga-Satake variety can be decomposed as $A\left(T_\Delta\right)\sim B^2$ where $B$ is an Abelian fourfold with $(-1,\Delta)_\Q$ contained in $\End_\Q(B)$.  Since $T_\Delta$ is a sub Hodge structure of $T(A_\Delta))$, the corresponding Kuga-Satake variety is a factor of $A(T(A_\Delta))\sim  A_\Delta^4$, i.e. $ B\sim A_\Delta$ so that $A(T_\Delta)\sim A^2_\Delta$. \qed \endproof
\begin{rmq}
The quaternion algebra $(-1,\Delta)_\Q$ has zero divisors  (i.e. $(-1,\Delta)_\Q\simeq M_2(\Q)$)   precisely when $\Delta$ is a sum of two squares  in $\Z$ which is the case if and only  if   all primes  $p\equiv 3 \bmod 4$ divide $\Delta$ with even power.   In these   cases the Abelian fourfold is isogeneous to a product $B^2$ with $B$ an Abelian  surface. It is an interesting open question if and how the  Kummer surface of $B$ and the K3 double plane are related.
This occurs for instance if $\Delta=1,2,4$ and for $\Delta=1$ we have a candidate for $B$. \end{rmq}

\subsection*{Acknowledgements}
We thank the anonymous referee for her/his detailed comments.
CP thanks the University of Turin as well as  the Riemann Center for Geometry and Physics of Leibniz Universit\"at Hannover for its  hospitality.
MS gratefully acknowledges support from ERC through StG 279723 (SURFARI).

\medskip

\noindent
Dipartimento di Matematica, Universit\`a di Torino, Via Carlo Alberto n.10, 10123 Torino, ITALY\\
giuseppe.lombardo@unito.it

\medskip

\noindent
Universit\'e de Grenoble I, D\'epartement de Math\'ematiques Institut Fourier, UMR 5582 du CNRS, 38402 Saint-Martin d'H\`eres Cedex, FRANCE\\
chris.peters@ujf-grenoble.fr

\medskip

\noindent Institut f\"{u}r Algebraische Geometrie, Leibniz Universit\"{a}t
  Hannover, Welfengarten 1, 30167 Hannover, GERMANY,\\
{schuett@math.uni-hannover.de}

\end{document}